\documentclass[namedate,webpdf]{ima-authoring-template}

\usepackage{amsmath}
\usepackage{stmaryrd}
\usepackage{physics}
\usepackage{braket}
\usepackage{bm}
\usepackage{subfigure}

\graphicspath{{figs/}}

\theoremstyle{thmstyletwo}%
\newtheorem{theorem}{Theorem}[section]%  meant for continuous numbers
\newtheorem{proposition}[theorem]{Proposition}%
\newtheorem{lemma}[theorem]{Lemma}%
\newtheorem{remark}{Remark}[section]%
\newtheorem{assumption}{Assumption}[section]%

\numberwithin{equation}{section}
\begin{document}

\DOI{10.1093/imanum/draf065}
\copyrightyear{2024}
\vol{00}
\pubyear{2024}
\access{Advance Access Publication Date: Day Month Year}
\appnotes{Paper}
\copyrightstatement{Published by Oxford University Press on behalf of the Institute of Mathematics and its Applications. All rights reserved.}
\firstpage{1}

%\subtitle{Subject Section}

\title[FEM semi-discrete error analysis of the DFN model]{Optimal convergence in finite element semi-discrete error analysis of the Doyle-Fuller-Newman model beyond 1D with a novel projection operator
}
 
\author{Shu Xu\ORCID{0009-0001-8492-1899}
    \address{\orgdiv{School of Mathematical Sciences}, \orgname{Peking University}, \orgaddress{\state{Beijing} \postcode{100871}, \country{China}} \\ \orgdiv{Institute of Computational Mathematics and Scientiﬁc/Engineering Computing, Academy of Mathematics and Systems Science}, \orgname{Chinese Academy of Sciences}, \orgaddress{\state{Beijing} \postcode{100190}, \country{China}}}
}

\author{Liqun Cao*
\address{\orgdiv{LSEC, NCMIS, Institute of Computational Mathematics and Scientiﬁc/Engineering Computing, Academy of Mathematics and Systems Science}, \orgname{Chinese Academy of Sciences}, \orgaddress{\state{Beijing} \postcode{100190}, \country{China}}}}

\authormark{S. Xu and L. Cao}

\corresp[*]{Corresponding author: \href{email:clq@lsec.cc.ac.cn}{clq@lsec.cc.ac.cn}}

\received{16}{11}{2024} 
\revised{11}{4}{2025}
\accepted{3}{6}{2025}

%\editor{Associate Editor: Name}

\abstract{We present a finite element semi-discrete error analysis for the Doyle-Fuller-Newman model, which is the most popular model for lithium-ion batteries. Central to our approach is a novel projection operator designed for the pseudo-($N$+1)-dimensional equation, offering a powerful tool for multiscale equation analysis. Our results bridge a gap in the analysis for dimensions $2 \le N \le 3$ and achieve optimal convergence rates of $h+\qty(\Delta r)^2$. Additionally, we perform a detailed numerical verification, marking the first such validation in this context. By avoiding the change of variables, our error analysis can also be extended beyond isothermal conditions.}

% \boxedtext{
% \begin{itemize}
% \item Key boxed text here.
% \item Key boxed text here.
% \item Key boxed text here.
% \end{itemize}}

\maketitle

\section{Introduction}
  The Doyle-Fuller-Newman (DFN) model \citep{doyle_DFN_modeling_1993, fuller_simulation_1994} is the most widely used physics-based model for lithium-ion batteries and is often referred to as a pseudo-two-dimensional (P2D) model when the battery region is simplified to one dimension.  It is essential in various engineering applications, including the calculation of battery state of charge (SOC), capacity performance under different operating conditions, and impedance spectra \citep{bloom_accelerated_2001,shi2015multi,smith_solid-state_2006,ramadesigan_modeling_2012,plett_BMS1_2015,hariharan_mathematical_2018,chen_porous_2022}.

  In the DFN model, a lithium-ion battery occupies a domain $\Omega \subset \mathbb{R}^N$, where $1 \le N \le 3$, and consists of three subdomains: the positive electrode $\Omega_\mathrm{p}$, the negative electrode $\Omega_\mathrm{n}$, and the separator  $\Omega_\mathrm{s}$. These domains form a laminated box structure, with $\bar \Omega = \bar \Omega_\mathrm{n} \cup \bar \Omega_\mathrm{s} \cup \bar \Omega_\mathrm{p}$, as shown in Fig.~\ref{fig:domain}.    
  Additionally, spherical particles with radii $R_\mathrm{p}$ and $R_\mathrm{n}$ are assumed to be uniformly distributed within each electrode and are parameterized by coordinates $\qty(x,r)\in \Omega_{2 r} := \qty(\Omega_\mathrm{n} \times (0, R_\mathrm{n})) \cup \qty(\Omega_\mathrm{p} \times (0, R_\mathrm{p}))$.
  The DFN model seeks to solve for the electrolyte potential $\phi_1(x,t)$, the electrode potential $\phi_2(x,t)$, the electrolyte lithium-ion concentration $c_1(x,t)$, and the lithium-ion concentration within particles $c_2(x,r,t)$. We define the electrolyte region as $\Omega_1 = \Omega_\mathrm{n} \cup \Omega_\mathrm{s} \cup \Omega_\mathrm{p}$ and the electrode region as $\Omega_2 = \Omega_\mathrm{n} \cup \Omega_\mathrm{p}$. Following a macrohomogeneous approach, the model assumes that the electrode and electrolyte phases coexist within $\Omega_2$. 

  \begin{figure}[htbp]
    \centering
    \includegraphics[width=0.5\textwidth]{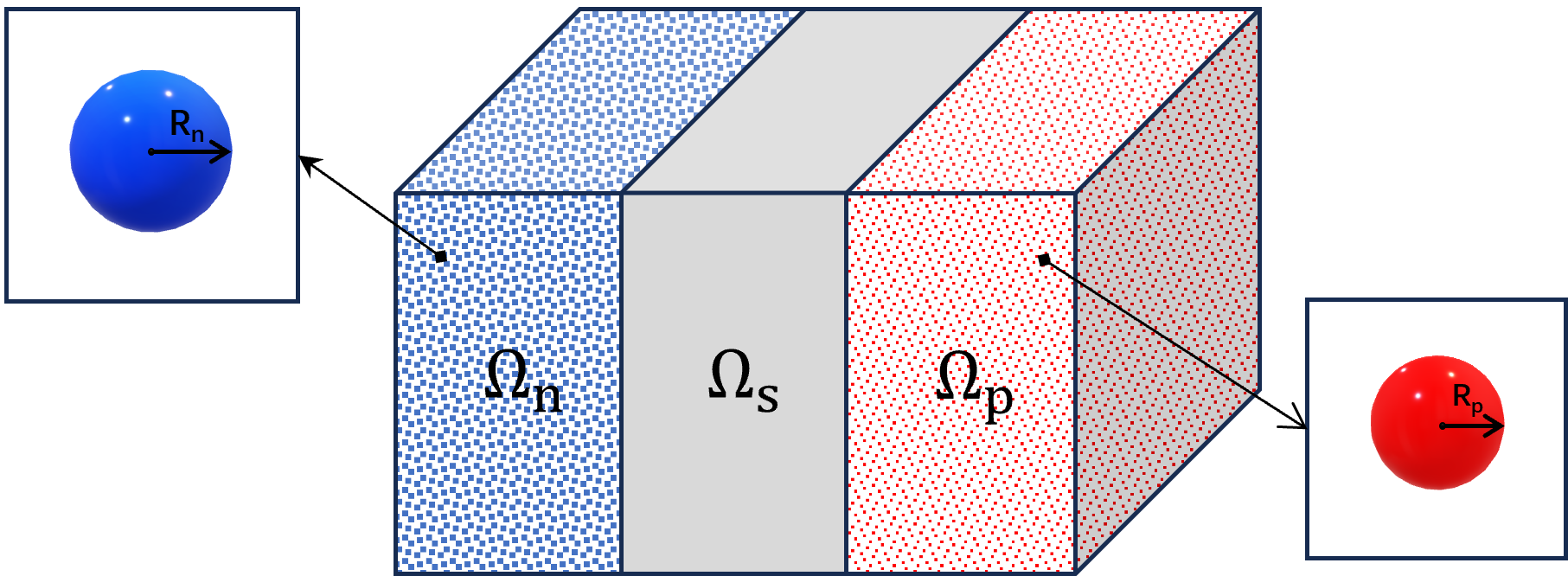}
    \caption{A 3D schematic representation of a Li-ion battery.}
    \label{fig:domain}
  \end{figure}

  Using the aforementioned notation, the final mathematical model is the following multi-domain, multi-scale and fully coupled nonlinear elliptic-parabolic system, consisting of three $N$-dimensional equations and one distinctive pseudo-($N$+1)-dimensional equation:
  \begin{equation}\label{eq:DFN_dfn_strong}
  \begin{cases}
  -\nabla \cdot\left(\kappa_1\qty(c_1) \nabla \phi_1 - \kappa_2\qty(c_1) \nabla f\qty(c_1)\right)=a_2 J\qty(c_1,\bar c_2, \phi_1, \phi_2), & (x, t) \in \Omega_1 \times\left(0, T\right), \\
  -\nabla \cdot\left(\sigma \nabla \phi_2\right)=-a_2 J\qty(c_1,\bar c_2, \phi_1, \phi_2), & (x, t) \in \Omega_2 \times\left(0, T\right), \\
  \varepsilon_1\frac{\partial c_1}{\partial t}-\nabla \cdot\left(k_1 \nabla c_1\right)=a_1 J\qty(c_1,\bar c_2, \phi_1, \phi_2), & (x, t) \in \Omega_1 \times\left(0, T\right), \\
  \frac{\partial c_2}{\partial t}-\frac{1}{r^2} \frac{\partial}{\partial r}\left(r^2 k_2 \frac{\partial c_2}{\partial r}\right)=0, & (x, r, t) \in \Omega_{2 r} \times\left(0, T \right),
  \end{cases}
  \end{equation} 
  with constraints
  \begin{equation}\label{eq:strong_constraints}
  c_1>0,\quad 0 < c_2 < c_{2,\max},
  \end{equation}
  and satisfies the initial conditions
  \begin{gather}\label{eq:strong_initial}
  c_1(x,0) = c_{10}(x) > 0, \quad x \in \Omega_1, \\ 
  c_2(x,r,0) = c_{20}(x,r),\quad 0 < c_{20}(x,r) < c_{2,\max}(x),\quad \qty(x,r) \in \Omega_{2r},
  \end{gather}
  boundary conditions
  \begin{gather}
      -\left.\left(\kappa_1 \nabla \phi_1 - \kappa_2 \nabla f\qty(c_1)\right) \cdot \boldsymbol{n}\right|_{\partial \Omega}=0, \\
      -\left.\sigma \nabla \phi_2 \cdot \boldsymbol{n}\right|_{\Gamma}=I,\quad -\left.\sigma \nabla \phi_2 \cdot \boldsymbol{n}\right|_{\partial \Omega_2 \setminus \Gamma}=0,\label{eq:strong_bdry_phi_2}\\
      -\left.k_1 \nabla c_1 \cdot \boldsymbol{n}\right|_{\partial \Omega} = 0, \\
      -r^2 k_2\left.\pdv{c_2(x)}{r}\right|_{r=0}= 0,\quad  -r^2 k_2\left.\pdv{c_2(x)}{r}\right|_{r=R_\mathrm{s}(x)}= \frac{J(x,c_1\qty(x),\bar c_2\qty(x), \phi_1\qty(x),  \phi_2\qty(x))}{F}, \quad x\in \Omega_2,\label{eq:strong_bdry_c_2}
  \end{gather}
  and the interface conditions 
  \begin{equation}\label{eq:DFN_strong_interface}
      \begin{aligned}
      \left. \llbracket \phi_1 \rrbracket \right|_{\Gamma_{\mathrm{s}k}} &= 0, \quad \left. \left\llbracket \qty(\kappa_1\nabla \phi_1  -\kappa_2 \nabla f(c_1)) \cdot \boldsymbol{\nu} \right\rrbracket\right|_{\Gamma_{\mathrm{s}k}}= 0, \quad &k \in \qty{\mathrm{n,p}},\\ 
      \left. \llbracket c_1 \rrbracket \right|_{\Gamma_{\mathrm{s}k}} &= 0,\quad \left. \left\llbracket k_1 \nabla c_1 \cdot \boldsymbol{\nu} \right\rrbracket \right|_{\Gamma_{\mathrm{s}k}}= 0 , \quad &k \in \qty{\mathrm{n,p}},
      \end{aligned}   
  \end{equation}
  where $\boldsymbol{n} $ is the outward unit normal vector, $\Gamma$ a measurable subset of $\partial \Omega_2$ with positive measure,  $\Gamma_\mathrm{sn}$ and $\Gamma_\mathrm{sp}$ the separator-negative and separator-positive interfaces respectively, $\left. \llbracket v \rrbracket \right|_{\Sigma}  $ the jump of a quantity $v$ across the interface $ \Sigma$ and $\boldsymbol{\nu} $ the outward unit normal vector at $\partial \Omega_\mathrm{s}$. The definition of all parameters and variables will be provided in section~\ref{sec:framework}. 

  It can be observed that the DFN model is a specific instance of our relatively more general mathematical formulation, when $\kappa_2 = \frac{2RT}{F}\kappa_1\qty(1-t^0_+)$, $f = \ln$, and $J$ is governed by the Butler-Volmer equation \citep{bermejo_implicit-explicit_2019}. For a detailed derivation of the DFN model, we recommend referring to \citet{fuller_simulation_1994, ciucci_derivation_2011, arunachalam_veracity_2015, timms_asymptotic_pouch_2021, richardson_charge_2022}. It is worth noting that we do not introduce the change of variables for $\phi_1$ (as used in prior studies such as \citet{wu_well-posedness_2006,diaz_well-posedness_2019,bermejo_numerical_2021}), as this approach is ineffective for non-uniform temperature distributions $T$ in $\kappa_2$. Consequently, some extra techniques are needed to address the coupling term $\nabla f\qty(c_1)$ in the first equation of \eqref{eq:DFN_dfn_strong}. 

  Research on the well-posedness of the DFN model dates back to \citet{wu_well-posedness_2006}, though this early work did not account for mass conservation equations at the particle scale. 
  Local existence and uniqueness of the solution for the full model were later established by \citet{kroener_mathematical_2016,diaz_well-posedness_2019}, with the solution extendable to a globally unique solution under specific conditions \citep{diaz_well-posedness_2019}. However, these studies focus only on the P2D model, leaving a gap in understanding for the actual P4D case. 
  Main mathematical challenges in this domain include strongly nonlinear source terms with singular behaviour, discontinuous and nonlinear coefficients, the lack of smoothness of the boundary and the pseudo-($N$+1)-dimensional equation. Related work can be found in \citet{seger_elliptic-parabolic_2013,ramos_well-posedness_2016,xu_life_2023,price_existence_2024}.

  Efficient simulation of the DFN model remains an active area of research.  
  Various spatial discretization methods, including the finite difference method \citep{fuller_simulation_1994,newman_electrochemical_2019,mao_finitedifference_1994,nagarajan_mathematical_1998}, finite volume method \citep{zeng_efficient_2013,mazumder_faster-than-real-time_2013}, finite element method \citep{bermejo_implicit-explicit_2019,bermejo_numerical_2021}, collocation method \citep{northrop_coordinate_2011,northrop_efficient_2015}, and hybrid method \citep{smith_solid-state_2006,kosch_computationally_2018} are applied to convert the system of PDEs to a system of DAEs. Computationally efficient time-stepping algorithms \citep{brenan_numerical_1995,bermejo_implicit-explicit_2019,korotkin_dandeliion_2021} are then used for full discretization, yielding a large number of algebraic equations with strong nonlinearity.   
  Efforts have also been made to design fast solvers \citep{wu_newton-krylov-multigrid_2002,bermejo_implicit-explicit_2019,han_fast_2021}, including the use of operator splitting \citep{farkas_improvement_2017} techniques, and to develop reduced-order models \citep{smith_control_2007,cai_reduction_2008,forman_reduction_2010,lass_pod_2013,landstorfer_modelling_2022} to improve simulation efficiency for control, online estimation and cell-design optimization applications.  

  Despite significant advancements in numerical methods, rigorous numerical analysis of this model remains limited  \citep{bermejo_numerical_2021}. 
  In \citet{bermejo_numerical_2021}, the first convergence analysis for the finite element discretization of this model in one dimension was presented. The core idea is the construction of an auxiliary approximator for the pseudo-($N$+1)-dimensional variable within the finite element space by composing a spatial barycentric interpolant with a radial projection.
  However, this approach encounters several limitations. 
  First, the well-definedness of the approximator is questionable: the range of the radial projection does not necessarily lie within the domain of the barycentric interpolation operator. Even under the $H^1$ regularity assumption imposed in that work, the lack of a continuous embedding $H^1\qty(\Omega) \hookrightarrow C\qty(\bar \Omega) $ in dimensions $N = 2,3$ renders the interpolation potentially ill-defined.
  Second, the error estimates based on this auxiliary approximator are neither direct nor easily tractable. This difficulty stems from the composition of two distinct operators, and that the barycentric interpolant is not bounded in the $L^2$ norm, which is crucial to his proof. 
  Lastly, the convergence rates proved with respect to $\Delta r$ are suboptimal, as they require specific arguments beyond the general embedding theory to accurately estimate the radial trace error.
  In addition, the study does not include numerical experiments to support its theoretical claims. 
  Thus, the present paper aims to address the aforementioned challenges and gaps, focusing on the finite element semi-discrete error analysis.

  The primary contributions of this paper are fourfold. First, we introduce a novel projection operator for the multiscale singular parabolic equation and derive its associated approximation error. Second, we establish convergence for finite element semi-discrete problems in $N$ dimensions ($1\le N\le3$), addressing prior limitations and demonstrating that the convergence rate is optimal with respect to both mesh size $h$ and $\Delta r$. Third, our model formulation allows for non-uniform temperature distributions, extending our error analysis straightforwardly to thermally coupled DFN models \citep{hariharan_mathematical_2018,hunt_derivation_2020,timms_asymptotic_pouch_2021}. Finally, we validate numerical convergence rates in both 2D and 3D cases using real battery parameters. To our knowledge, this verification has not been previously reported.

  This paper is organized as follows: Section~\ref{sec:framework} introduces the weak formulation and foundational assumptions for the mathematical model. Section~\ref{sec:fem_semi} proposes the novel projection operators and presents the finite element error estimates with optimal convergence rates for the semi-discrete problem.  In section~\ref{sec:experiments}, we confirm the accuracy of these estimates through numerical experiments using real battery parameters. %Finally, section~\ref{sec:conclusions} provides concluding remarks.

\section{Mathematical framework}\label{sec:framework}

  In this section, all data in \eqref{eq:DFN_dfn_strong}-\eqref{eq:DFN_strong_interface} will be clarified. First, we introduce some notation and function spaces. After imposing suitable assumptions for battery modeling, the weak formulation is established.

  \subsection{Notation and function spaces}

    $L^p(\Omega)(1\le p \le \infty)$ and $H^m(\Omega)(m \in \mathbb{N}_0)$ denote Lebesgue and Sobolev spaces defined on $\Omega$ respectively. We also denote the subspace of $H^1\qty(\Omega)$, consisting of functions whose integral is zero, by $H^1_*\qty(\Omega) := \qty{ v \in H^1\qty(\Omega):\: \int_\Omega v\, \dd x = 0}$.
    % \begin{displaymath}
    %     H^1_*\qty(\Omega) := \qty{ v \in H^1\qty(\Omega):\: \int_\Omega v\, \dd x = 0}.
    % \end{displaymath}
    Since multiple domains $\Omega_\mathrm{n}$, $\Omega_\mathrm{s}$ and $\Omega_\mathrm{p}$ are involved, it is necessary to further define piecewise Sobolev spaces 
    \begin{displaymath}
        H^2_\mathrm{pw}\qty(\Omega_1) =  H^2\qty(\Omega_\mathrm{n}) \cap H^2\qty(\Omega_\mathrm{s}) \cap H^2\qty(\Omega_\mathrm{p})
    \end{displaymath}
    equipped with the norm $\norm{\cdot}_{H^2_\mathrm{pw}\qty(\Omega_1) }:= \norm{\cdot}_{H^2\qty(\Omega_\mathrm{n}) } + \norm{\cdot}_{H^2\qty(\Omega_\mathrm{s}) } + \norm{\cdot}_{H^2\qty(\Omega_\mathrm{p}) }$,
    % \begin{displaymath}
    %     \norm{\cdot}_{H^2_\mathrm{pw}\qty(\Omega_1) }:= \norm{\cdot}_{H^2\qty(\Omega_\mathrm{n}) } + \norm{\cdot}_{H^2\qty(\Omega_\mathrm{s}) } + \norm{\cdot}_{H^2\qty(\Omega_\mathrm{p}) }
    % \end{displaymath}
    and 
    \begin{displaymath}
        H^2_\mathrm{pw}\qty(\Omega_2) =   H^2\qty(\Omega_\mathrm{n}) \cap H^2\qty(\Omega_\mathrm{p})
    \end{displaymath}
    with $\norm{\cdot}_{H^2_\mathrm{pw}\qty(\Omega_2) }:= \norm{\cdot}_{H^2\qty(\Omega_\mathrm{n}) } + \norm{\cdot}_{H^2\qty(\Omega_\mathrm{p}) }$.
    % \begin{displaymath}
    %     \norm{v}_{H^2_\mathrm{pw}\qty(\Omega_2) }:= \norm{v}_{H^2\qty(\Omega_\mathrm{n}) } + \norm{v}_{H^2\qty(\Omega_\mathrm{p}) }.
    % \end{displaymath}

    The natural space for radial solutions with radial coordinate $r$ is $L^2_r\qty(0,R)$, consisting of measurable functions $v$ defined on $\qty(0,R)$ such that $vr$ is $L^2$-integrable. It is obvious that $L^2_r\qty(0,R)$ is a Hilbert space and can be endowed with the norm 
    \begin{displaymath}
        \norm{v}_{L^2_r\qty(0,R)} = \qty( \int_{0}^R \qty|v(r)|^2r^2 \, \dd r )^\frac{1}{2}.
    \end{displaymath}
    We also denote by $H^m_r\qty(0,R)(m \in \mathbb{N}_0)$ the Hilbert spaces of measurable functions $v$, whose distribution derivatives belong to $L^2_r\qty(0,R)$ up to order $m$, with the norm 
    \begin{equation}
        \norm{v}_{H^m_r\qty(0,R)} = \qty( \sum_{k=0}^{m} \norm{\dv[k]{v}{r}}^2_{L^2_r\qty(0,R)})^\frac{1}{2}.
    \end{equation}
    Besides, it is convenent for pseudo-($N$+1)-dimensional variables to define 
    \begin{displaymath}
        H^p\qty(\Omega_2 ; H^q_r \qty(0,R_\mathrm{s}(\cdot))):= H^p\qty(\Omega_\mathrm{n} ; H^q_r \qty(0,R_\mathrm{n})) \cap H^p\qty(\Omega_\mathrm{p} ; H^q_r \qty(0,R_\mathrm{p})),\quad p,q \in \mathbb{N}_0.
    \end{displaymath} 
    For functions in $H^1_r\qty(0,R)$, we have the following critical trace estimation in \citet{bermejo_numerical_2021}:
    \begin{proposition}\label{prop:DFN_radial_surface}
        There exist an arbitrarily small number $\epsilon$ and a positive (possibly large) constant $C(\epsilon)$, such that for $u \in H^1_r\qty(0,R)$,
        \begin{displaymath}  
            \left|u\left(R\right)\right| 
            \leq \epsilon\left\|\frac{\partial u}{\partial r}\right\|_{L_r^2\left(0, R\right)}  + 
            C( \epsilon)\left\|u\right\|_{L_r^2\left(0, R\right)}.
        \end{displaymath}
    \end{proposition}

  \subsection{Definition and assumptions on the data} 
    In \eqref{eq:DFN_dfn_strong}, except the function $f\colon (0,+\infty) \rightarrow \mathbb{R}$, all the remaining material data is piecewise, i.e.,
    \begin{equation*}
        \kappa_i = \sum_{m\in \set{\mathrm{n,s,p}}}\kappa_{im}\qty(c_1)\boldsymbol{1}_{\Omega_m}, \quad a_i =\sum_{m\in \set{\mathrm{n,p}}} a_{im}\boldsymbol{1}_{\Omega_m},\quad i=1,2, \quad c_{2,\max}  =\sum_{m\in \set{\mathrm{n,p}}}c_{2,\max,m}\boldsymbol{1}_{\Omega_m},
    \end{equation*}
    \begin{equation*}
        \varepsilon_1 =\sum_{m\in \set{\mathrm{n,s,p}}}\varepsilon_{1m}\boldsymbol{1}_{\Omega_m},\quad k_1 =\sum_{m\in \set{\mathrm{n,s,p}}}k_{1m}\boldsymbol{1}_{\Omega_m}, \quad   \sigma =\sum_{m\in \set{\mathrm{n,p}}}\sigma_{m}\boldsymbol{1}_{\Omega_m},\quad k_2 = \sum_{m\in \set{\mathrm{n,p}}}k_{2m}\boldsymbol{1}_{\Omega_m},
    \end{equation*}
    where $\boldsymbol{1}_{\Omega_m}$ is the indicator function for $\Omega_m$, $a_{1m}$, $a_{2m}$, $c_{2,\max,m}$, $\varepsilon_{1m}$, $k_{1m}$, $\sigma_m$, $k_{2m}$   are positive constants,
    \begin{equation*}
        \kappa_{1m},\,\kappa_{2m}\colon (0,+\infty) \rightarrow (0,+\infty).
    \end{equation*}  
    As to the source term $J$, setting $\bar c_2(x,t) = c_2\qty(x,R_\mathrm{s}(x),t)$ and $\eta = \phi_2 - \phi_1 - U$ with $R_\mathrm{s}  =\sum_{m\in \set{\mathrm{n,p}}}R_{m}\boldsymbol{1}_{\Omega_m}$, $U = \sum_{m\in \set{\mathrm{n,p}}}U_m(c_1,\bar c_2)\boldsymbol{1}_{\Omega_m}$ and $U_m \colon (0,+\infty) \times \qty(0,c_{2,\mathrm{max},m})  \rightarrow \mathbb{R}$,
    % \begin{gather}
    %      U = \sum_{m\in \set{\mathrm{n,p}}}U_m(c_1,\bar c_2)\boldsymbol{1}_{\Omega_m}, \\ 
    %      \bar c_2(x,t) = c_2\qty(x,R_\mathrm{s}(x),t),\quad R_\mathrm{s}  =\sum_{m\in \set{\mathrm{n,p}}}R_{m}\boldsymbol{1}_{\Omega_m},\quad
    % \end{gather}
    we suppose 
    \begin{equation*}
        J =\sum_{m\in \set{\mathrm{n,p}}}J_{m}\qty(c_1,\bar c_2,\eta)\boldsymbol{1}_{\Omega_m},
    \end{equation*}
    where 
    $J_m \colon (0,+\infty) \times \qty(0,c_{2,\mathrm{max},m}) \times \mathbb{R}  \rightarrow \mathbb{R}$.  Besides, there is a given function $I\colon \Gamma \rightarrow \mathbb{R}$ in \eqref{eq:strong_bdry_phi_2}, and $F$ is a positive constant in \eqref{eq:strong_bdry_c_2}. 
    We set $\underline{v} = \min_x v$ and $\bar{v} = \max_x{v}$, when $v$ is piecewise constant. 
   
    %Our final formulation in this paper focuses on cases where the following conditions are met:
    \begin{assumption}\label{asp:IB_condition}
        $c_{10}\in L^2\qty(\Omega)$, 
        $c_{20} \in L^2\qty(\Omega_2 ; L^2_r \qty(0,R_\mathrm{s}(\cdot)))$.
        $c_{10}(x) > 0, \:x\in \Omega_1\; \mathrm{a.e.}$;
        $ 0 < c_{20}(x,r) < c_{2,\max}(x),\:(x,r) \in \Omega_{2r}\; \mathrm{a.e.}$.
        $I \in L^2\qty(0,T; L^2\qty(\Gamma)).$
    \end{assumption}

    \begin{assumption}\label{asp:nonlinear_function}
        $f \in C^1\qty(0,+\infty)$. 
        $\kappa_{1m},\kappa_{2m} \in C^2\qty((0,+\infty))$,  $ U_m \in C^2\qty((0,+\infty) \times \qty(0,c_{2,\mathrm{max},m}) )$, $J_m \in C^1\qty((0,+\infty) \times \qty(0,c_{2,\mathrm{max},m})  \times \mathbb{R}  )$, $m\in\set{\mathrm{n,p}}$ and 
        \begin{equation*}
            \exists C_0>0\colon \forall \qty(c_1,\bar c_2, \eta) \in (0,+\infty) \times \qty(0,c_{2,\mathrm{max},m})  \times \mathbb{R}, \quad   \pdv{ J_m}{\eta}\qty(c_1,\bar c_2, \eta) \ge C_0. 
        \end{equation*}
    \end{assumption}

  \subsection{Definition of weak solution}
    \label{subsec:framework_weak_sol}
    Given $T>0$ and setting $\Omega_{1T} = \Omega_1 \times \qty(0,T)$, $\Omega_{2rT} = \Omega_{2r} \times \qty(0,T)$, we define a weak solution of \eqref{eq:DFN_dfn_strong}-\eqref{eq:DFN_strong_interface} as a quadruplet $\left(\phi_1, \phi_2, c_1, c_2\right)$,
    \begin{gather}
    \phi_1 \in L^2\left(0, T; H^1_*(\Omega) \right), \quad
    \phi_2 \in L^2\left(0, T; H^1\left(\Omega_2\right)\right),\\
    c_1 \in C\left([0, T); L^{2}(\Omega)\right) \cap L^2\left(0, T; H^1(\Omega)\right), \quad \pdv{c_{1}}{t}  \in L^{2}\left(0, T; L^2(\Omega)\right),  \\
    c_2 \in C\left([0, T); L^2\qty(\Omega_2;L_r^2\qty(0,R_\mathrm{s}\qty(\cdot)))\right) \cap L^2\left(0, T; L^2\qty(\Omega_2;H_r^1\qty(0,R_\mathrm{s}\qty(\cdot))) \right), \\ 
    \pdv{c_{2}}{t}  \in L^{2}\left(0, T;  L^2\qty(\Omega_2;L_r^2\qty(0,R_\mathrm{s}(\cdot)))\right), 
    \end{gather}
    such that  $c_1(0) = c_{10}$, $c_2(0) = c_{20}$, $f\qty(c_1) \in L^2\qty(0,T;H^1\qty(\Omega))$, $J \in L^2\qty(0,T;L^2\qty(\Omega_2))$, $\kappa_1,\, \kappa_2 \in L^\infty\qty(\Omega_{1T})$, $c_{1}(x,t)>0$, $(x,t) \in \Omega_{1T}\;\mathrm{a.e.}$, $ 0 < c_{2}(x,r,t) < c_{2,\max}(x),\,(x,r,t)\in \Omega_{2rT}\;\mathrm{a.e.}$ and for $t \in(0, T)\;\mathrm{ a.e.}$, 
    \begin{equation}\label{eq:DFN_weak_phi1}
        \int_{\Omega} \kappa_1\qty(t) \nabla \phi_1\qty(t) \cdot \nabla \varphi\, \dd x - \int_{\Omega} \kappa_2\qty(t) \nabla f\qty(c_1\qty(t)) \cdot \nabla \varphi \, \dd x -\int_{\Omega_2} a_2 J\qty(t) \varphi \, \dd x = 0,\, 
        \forall \varphi \in H_*^{1}\left(\Omega\right),
    \end{equation}
    \begin{equation}\label{eq:DFN_weak_phi2}
        \int_{\Omega_2} \sigma \nabla \phi_2\qty(t) \cdot \nabla \varphi \, \dd x +\int_{\Omega_2}a_2 J\qty(t)\varphi \, \dd x + \int_{\Gamma}^{} I\qty(t) \varphi\, \dd s =0,\, \forall \varphi \in H^{1}\left(\Omega_2\right),
    \end{equation}
    \begin{equation}\label{eq:DFN_weak_c1}
        \int_{\Omega} \varepsilon_1 \dv{c_1}{t}\qty(t) \varphi \, \dd x +\int_{\Omega} k_1 \nabla c_1\qty(t) \cdot \nabla \varphi \, \dd x - \int_{\Omega_2} a_1 J\qty(t) \varphi \, \dd x =0, \, \forall \varphi \in H^{1}\left(\Omega\right)
    \end{equation}
    \begin{multline}\label{eq:DFN_weak_c2}
        \int_{\Omega_2} \int_0^{R_\mathrm{s}(x)} \dv{c_2}{t}\qty(t) \psi r^2 \, \dd r\dd x + \int_{\Omega_2} \int_0^{R_\mathrm{s}(x)} k_2 \frac{\partial c_2\qty(t)}{\partial r} \frac{\partial \psi}{\partial r} r^2 \;\dd r \dd x + \int_{\Omega_2} \frac{R_\mathrm{s}^2(x)}{F} J(x,t) \psi\left(x,R_\mathrm{s}(x)\right) \dd x =0, \\
        \forall \psi \in L^2\qty(\Omega_2;H_r^{1}\left(0, R_\mathrm{s}(\cdot)\right)).
    \end{multline}

    Under a slightly more stringent version of Assumptions~\ref{asp:IB_condition} and \ref{asp:nonlinear_function} with higher regularity, there is a unique weak solution when $T$ being small enough \citep{kroener_mathematical_2016,diaz_well-posedness_2019}.

\section{Error analysis of finite element semi-discrete problems}\label{sec:fem_semi}

  In this section, we only consider $\bar\Omega_{m}\subset \mathbb{R}^N,\; m\in\set{\mathrm{n, s, p}}$ as polygonal domains such that $\bar\Omega_{m}$ is the union of a finite number of polyhedra. Let $\mathcal{T}_{h,m}$ be a regular family of  triangulation for $\bar\Omega_{m},\; m\in\set{\mathrm{n, s, p}}$, and the meshes match at the interfaces, i.e., the points, edges and faces of discrete elements fully coincide at the interface. Hence, $\mathcal{T}_{h} := \bigcup_{m\in\set{\mathrm{n, s, p}}} \mathcal{T}_{h,m}$ is also a regular family of triangulation for $\bar \Omega$. Additionally, it is assumed that all $(K,P_K,\Sigma_K)$, $K\in\mathcal{T}_{h}$, are finite element affine families \citep{ciarlet2002finite}. For the closed interval $[0,R_\mathrm{s}(x)]$, $x\in \Omega_2$, let $\mathcal{T}_{\Delta_r}(x)$ be a family of regular meshes, such that $[0,R_\mathrm{s}(x)] = \bigcup_{I\in\mathcal{T}_{\Delta_r}(x)}I$. It is also assumed that each element of $\mathcal{T}_{\Delta_r}(x)$ is affine equivalent to a reference element. Since $R_\mathrm{s}$ is a piecewise constant, here only two sets of meshes are considered, i.e.,
  \begin{equation*}
      \mathcal{T}_{\Delta_r}(x)= \left\{\begin{array}{l}\mathcal{T}_{\Delta_r \mathrm{p}}\qcomma{} x\in\Omega_\mathrm{p}, \\ \mathcal{T}_{\Delta_r \mathrm{n}}\qcomma{} x\in\Omega_\mathrm{n}.  \end{array} \right. 
  \end{equation*}
  For each element $K\in \mathcal{T}_{h}$ and $I \in \cup_{x\in \Omega_2} \mathcal{T}_{\Delta_r}(x)$, we use $h_K$ and $\Delta r_I$ for the diameter respectively. Let $h = \max_{K\in \mathcal{T}_{h}}  h_K $ and $\Delta r = \max_{I\in \cup_{x\in \Omega_2} \mathcal{T}_{\Delta_r}(x)}  \Delta r_I$.

  The unknowns $\phi_1$, $\phi_2$ and $c_1$ are discretized by piecewise-linear elements. Let $V_h^{(1)}\left(\bar \Omega\right)$ and $V_h^{(1)}\left(\bar \Omega_2\right)$ be the corresponding finite element spaces for $c_1$ and $\phi_2$, while 
  \begin{equation*}
    W_h\left(\bar \Omega\right)= \qty{w_h \in V_h^{(1)}\left(\bar \Omega\right):\: \int_{\Omega} w_h(x) \dd x = 0},
  \end{equation*}
  for $\phi_1$. 
  For the pseudo-($N$+1) dimensional $c_2$, piecewise-constant elements are used for discretization in the $x$ coordinate while piecewise-linear elements in the $r$ coordinate. Namely, a tensor product finite element space 
  \begin{equation}
    V_{h \Delta r}\left(\bar \Omega_{2 r}\right):=  \qty(V_h^{(0)}\left(\bar \Omega_\mathrm{n}\right) \otimes V_{\Delta r}^{(1)}\left[0,R_\mathrm{n}\right]) \bigcap \qty(V_h^{(0)}\left(\bar \Omega_\mathrm{p}\right) \otimes V_{\Delta r}^{(1)}\left[0,R_\mathrm{p}\right])
  \end{equation}
  is used. Notice that the spatial mesh for $c_2$ is the same as that for $\phi_2$, whereas a dual mesh is used in \citet{bermejo_numerical_2021}.
  Then, we propose finite element semi-discretization for the problem in subsection~\ref{subsec:framework_weak_sol} as follows:

  Given $\left(c_{1 h}^0, c_{2 h \Delta r}^0\right) \in V_h^{(1)}\left(\bar \Omega\right) \times V_{h \Delta r}\left(\bar \Omega_{2 r}\right)$, $c_{1 h}^0(x) > 0$, $0 < c_{2 h \Delta r}^0(x,r) < c_{2,\max}(x)$, for each $t\in \qty[0,T]$, find $\left(\phi_{1 h}(t), \phi_{2 h}(t), c_{1h}(t), c_{2 h\Delta r}(t)\right)$  $\in W_h\left(\bar \Omega\right) \times V_h^{(1)}\left(\bar \Omega_2\right) \times V_h^{(1)}\left(\bar \Omega\right) \times V_{h \Delta r}\left(\bar \Omega_{2 r}\right),$ such that $c_{1h}\qty(0)=c_{1 h}^0$, $c_{2h\Delta r}\qty(0)=c_{2 h\Delta r}^0$, $c_{1h}(x)> 0$, $ 0 < c_{2h\Delta r}(x,r) < c_{2,\max}(x)$, and 
  \begin{equation}\label{eq:DFN_fem_phi1}
      \int_{\Omega} \kappa_{1h}\qty(t)\nabla \phi_{1h}\qty(t) \cdot \nabla w_h\, \dd x - \int_{\Omega} \kappa_{2h}\qty(t)\nabla f\qty(c_{1h}\qty(t)) \cdot \nabla w_h\,\dd x - \int_\Omega a_2 J_h\qty(t) w_h \,\dd x =0,\, \forall w_h \in W_h\left(\bar \Omega\right),
  \end{equation}
  \begin{equation}\label{eq:DFN_fem_phi2}
      \int_{\Omega_2} \sigma \nabla \phi_{2 h}\qty(t) \cdot \nabla v_{2 h}\, \dd x +\int_{\Omega_2} a_2 J_h\qty(t) v_{2h}\, \dd x + \int_\Gamma I\qty(t) v_{2h} \,\dd x =0,\quad \forall v_{2h} \in V_h^{(1)}\left(\bar \Omega_2\right),
  \end{equation}
  \begin{equation}\label{eq:DFN_fem_c1}
      \int_{ \Omega } \varepsilon_1 \dv{c_{1 h}}{t}\qty(t) v_{1h}\, \dd x +\int_{\Omega} k_{1} \nabla c_{1 h}\qty(t) \cdot \nabla v_{1h} \, \dd x-\int_{\Omega_2} a_1 J_h\qty(t) v_{1h}\, \dd x=0,\quad \forall v_{1h} \in V_h^{(1)}\left(\bar \Omega\right),
  \end{equation}
  \begin{multline}\label{eq:DFN_fem_c2}
      \int_{\Omega_2} \int_0^{R_\mathrm{s}(x)} \dv{c_{2 h \Delta r}}{t} \qty(t) v_{h \Delta r} r^2\, \dd r \dd x+\int_{\Omega_2} \int_0^{R_\mathrm{s}(x)} k_{2} \frac{\partial c_{2 h \Delta r}\qty(t)}{\partial r} \frac{\partial v_{h \Delta r}}{\partial r} r^2 \, \dd r \dd x \\
      +\int_{\Omega_2} \frac{R_\mathrm{s}^2(x)}{F} J_h(x,t) v_{h \Delta r}\left(x, R_\mathrm{s}(x)\right)\, \dd x = 0, \quad \forall v_{h \Delta r} \in V_{h \Delta r}\left(\bar \Omega_{2 r}\right),
  \end{multline}
  where $\kappa_{ih}\qty(t) = \sum_{m\in \set{\mathrm{n,s,p}}}\kappa_{im}\qty(c_{1h}\qty(t))\boldsymbol{1}_{\Omega_m}$, $i=1,2$, $J_h\qty(t) =\sum_{m\in \set{\mathrm{n,p}}}J_{m}\qty(c_{1h}\qty(t),\bar c_{2h}\qty(t),\eta_{h}\qty(t))\boldsymbol{1}_{\Omega_m}$, 
  with $\bar{c}_{2 h}(x,t)=c_{2 h \Delta r}\qty(x, R_\mathrm{s}(x), t)$, $U_h\qty(t) = \sum_{m\in \set{\mathrm{n,p}}} U_m(c_{1h}\qty(t),\bar c_{2h}\qty(t))\boldsymbol{1}_{\Omega_m}$, and $\eta_{h}\qty(t) = \phi_{2h}\qty(t) - \phi_{1h}\qty(t) - U_h\qty(t)$.

  To present finite element error analysis, we make the following assumptions:
  \begin{assumption}\label{asp:DFN_fem_IC}
    $ c_{10} \in  H^2_\mathrm{pw}\qty(\Omega_1),\, c_{20} \in  H^1\left(\Omega_2 ; H_r^1(0, R_\mathrm{s}\qty(\cdot))\right) \cap L^2\left(\Omega_2 ; H_r^2\qty(0, R_\mathrm{s}\qty(\cdot))\right)$.
    %   \begin{gather}
    %       c_{10} \in  H^2_\mathrm{pw}\qty(\Omega_1),\, c_{20} \in  H^1\left(\Omega_2 ; H_r^1(0, R_\mathrm{s}\qty(\cdot))\right) \cap L^2\left(\Omega_2 ; H_r^2\qty(0, R_\mathrm{s}\qty(\cdot))\right). 
    %   \end{gather} 
      % $c_{10} \in  H^2_\mathrm{pw}\qty(\Omega_1)$, $c_{20} \in  H^1\left(\Omega_2 ; H_r^1(0, R_\mathrm{s}\qty(\cdot))\right) \cap L^2\left(\Omega_2 ; H_r^2\qty(0, R_\mathrm{s}\qty(\cdot))\right)$.
  \end{assumption}
  \begin{assumption}\label{asp:DFN_fem_regularity}
      $\phi_1 \in L^2\left(0, T ; H_\mathrm{p w}^2\left(\Omega_1\right)\right)$,  
      $\phi_2 \in L^2\left(0, T ; H_\mathrm{p w}^2\left(\Omega_2\right)\right)$,  
      $c_1 \in H^1\left(0, T ; H_\mathrm{p w}^2\left(\Omega_1\right)\right)$, 
      $c_2 \in H^1\left(0, T ; H^1\left(\Omega_2 ; H_r^1(0, R_\mathrm{s}\qty(\cdot))\right) \cap L^2\left(\Omega_2 ; H_r^2\qty(0, R_\mathrm{s}\qty(\cdot))\right)\right)$.
%   \begin{gather}
%       \phi_1 \in L^2\left(0, T ; H_\mathrm{p w}^2\left(\Omega_1\right)\right), \phi_2 \in L^2\left(0, T ; H_\mathrm{p w}^2\left(\Omega_2\right)\right),  
%       c_1 \in H^1\left(0, T ; H_\mathrm{p w}^2\left(\Omega_1\right)\right), \\
%       c_2 \in H^1\left(0, T ; H^1\left(\Omega_2 ; H_r^1(0, R_\mathrm{s}\qty(\cdot))\right) \cap L^2\left(\Omega_2 ; H_r^2\qty(0, R_\mathrm{s}\qty(\cdot))\right)\right),
%   \end{gather}
  \end{assumption}
  \begin{assumption}\label{asp:DFN_prior_bound}
  There exist constants $L,M,N,Q>0$, such that for $t\in \qty[0,T]$ a.e.,
  \begin{gather*}
      \left\|\phi_i(t)\right\|_{L^{\infty}\left(\Omega_i\right)},\,
      \left\| \nabla \phi_i(t)\right\|_{L^{\infty}\left(\Omega_i\right)} \leq L, \; i=1,2, \\ 
      \frac{1}{M} \leq c_1(t)\leq M, \quad 
      \norm{\nabla c_1(t)}_{L^{\infty}\left(\Omega\right)}\le Q, \quad
      \frac{1}{N} \leq \frac{\bar c_2(t)}{c_{2,\max }} \leq\left(1-\frac{1}{N}\right).
  \end{gather*}
  \end{assumption}
  \begin{remark}
    Without the previously mentioned change of variables, it becomes necessary to explicitly assume the $L^\infty$ boundedness of $\nabla c_1$.
  \end{remark}

  \begin{assumption}\label{asp:DFN_fem_prior_bound}
    There exist constants $L,M,N>0$, such that for $t\in \qty[0,T]$ a.e.,
    \begin{gather*}
        \left\|\phi_{i h}(t)\right\|_{L^{\infty}\left(\Omega_i\right)} \leq L, \; i=1,2, \quad 
        \frac{1}{M} \leq  c_{1h}(t) \leq M, \quad 
        \frac{1}{N} \leq  \frac{\bar c_{2h}(t)}{c_{2,\max }} \leq\left(1-\frac{1}{N}\right).
    \end{gather*}
  \end{assumption}

  \subsection{Pseudo-(N+1)-dimensional projection approximation}
    
    To analyze the finite element error of the multiscale pseudo-($N$+1)-dimensional equation, we need to introduce several projection operators into tensor product spaces and their approximation properties.
    
    In this subsection, let $\Omega$ represent a generic domain and $R$ a positive constant. Denote by $\mathcal{T}_h$ the triangulation of $\bar\Omega$ and by $\mathcal{T}_{\Delta_r}(x)$ the triangulation of the closed interval $[0, R_\mathrm{s}]$. Define the corresponding finite element spaces $V_h^{(0)}\left(\bar \Omega\right)$, $V_h^{(1)}\left(\bar \Omega\right)$ and $V_{\Delta r}^{(1)}\qty[0, R]$ as the spaces of piecewise constant functions, piecewise linear functions over $\bar \Omega$, and piecewise linear functions over $[0, R_\mathrm{s}]$, respectively.
    % First, define the function space
    % $$H_{\enspace r}^{p, q}\left(\Omega \times(0, R)\right)=H^p\left(\Omega; L_r^2(0, R)\right) \cap L^2\left(\Omega; H_r^q(0, R)\right).$$

    Given a constant $\lambda>0$ and $b \in L^\infty\qty(\Omega\times\qty(0,R))$, such that $0 <\underline{b} \le b \le \bar b $ a.e., for constants $\underline{b}$ and $\bar b$. 
    Define the projection operator $P_{h\Delta r}: L^2\left(\Omega ; H_r^1\left(0, R\right)\right) \rightarrow V_h^\qty(0)\left(\bar \Omega\right) \otimes$ $V_{\Delta r}^{(1)}\qty[0, R]$, such that
    \begin{equation}\label{eq:DFN_c2_projection_hr}
        \int_{\Omega} \int_0^{R}\left(b \frac{\partial\left(P_{h\Delta r} w-w\right)}{\partial r} \pdv{w_{h\Delta r}}{r}+\lambda\left(P_{h\Delta r} w-w\right) w_{h\Delta r}\right) r^2 \dd r \dd x = 0, \quad
        \forall w_{h\Delta r} \in V_h^\qty(0)\left(\bar \Omega\right) \otimes V_{\Delta r}^{(1)}\qty[0, R] ,
    \end{equation} 
    the projection operator $P_{h}:  L^2\left(\Omega ; H_r^1\left(0, R\right)\right) \rightarrow V_h^\qty(0)\left(\bar \Omega\right) \otimes H_{r}^{1}\qty(0, R)$, such that
    \begin{equation}\label{eq:DFN_c2_projection_h}
        \int_{\Omega} \int_0^{R}\left(b \frac{\partial\left(P_{h} w-w\right)}{\partial r} \pdv{w_{h}}{r}+\lambda\left(P_{h} w-w\right) w_{h}\right) r^2 \dd r \dd x = 0, \quad
        \forall w_{h} \in V_h^\qty(0)\left(\bar \Omega\right) \otimes H_{r}^{1}\qty(0, R) ,
    \end{equation} 
    and the projection operator $P_{\Delta r}: L^2\left(\Omega ; H_r^1\left(0, R\right)\right)\rightarrow L^2\left(\Omega\right) \otimes V_{\Delta r}^{(1)}\qty[0, R]$, such that
    \begin{equation}\label{eq:DFN_c2_projection_r}  
        \int_{\Omega} \int_0^{R}\left(b \frac{\partial\left(P_{\Delta r} w-w\right)}{\partial r} \pdv{w_{\Delta r}}{r}+\lambda\left(P_{\Delta r} w-w\right) w_{\Delta r}\right) r^2 \dd r \dd x = 0, \quad
        \forall w_{\Delta r} \in L^2\left(\Omega\right) \otimes V_{\Delta r}^{(1)}\qty[0, R] .
    \end{equation} 
    It is obvious that all above projection operators are well-defined and $P_{h\Delta r} = P_h P_{\Delta r} =P_{\Delta r} P_h  $. What we are concerned with are their approximation properties.
    
    \begin{lemma}\label{lemma:DFN_projection_dr}
        If $w\in L^2\left(\Omega ; H_r^2\left(0, R\right)\right)$, then for $q=0,1$, we have
        \begin{equation}\label{eq:DFN_estimation_Pr}
            \left\|w-P_{\Delta r} w\right\|_{ L^2\left(\Omega ; H_r^q\left(0, R\right)\right)} \leq C \qty(\Delta r)^{2-q}\|w\|_{ L^2\left(\Omega ; H_r^2\left(0, R\right)\right)} . 
        \end{equation}
    \end{lemma}

    \begin{proof}
      We prove that $P_{\Delta r} w(x)$ is also the $H^1_r$-orthogonal projection of $w(x) \in H^1_r(0,R)$  on $V_{\Delta r}^{(1)}\qty[0, R]$, i.e., %for $x \in \Omega\, \mathrm{ a.e.}$, 
      \begin{equation}\label{eq:DFN_r_orthogonal}
          \int_0^{R}\left(b(x) \frac{\partial\left(P_{\Delta r} w(x)-w(x)\right)}{\partial r} \frac{d v_{\Delta r}}{d r}+\lambda\left(P_{\Delta r} w-w\right)(x) v_{\Delta r}\right) r^2 \dd r = 0, 
          \forall v_{\Delta r} \in  V_{\Delta r}^{(1)}\qty[0, R], \, x \in \Omega \, \mathrm{a.e.} .
      \end{equation}   
        Take the test function $ w_{\Delta r}= \phi v_{\Delta r}$ with $\phi\in L^2\qty(\Omega)$ and $v_{\Delta r}\in V_{\Delta r}^{(1)}\qty[0, R]$ in \eqref{eq:DFN_c2_projection_r}. Since the left-hand side of \eqref{eq:DFN_r_orthogonal} is in $ L^2\qty(\Omega)$,  \eqref{eq:DFN_r_orthogonal} follows due to the arbitrariness of $\phi$. 
        
      Then just following the same argument in \citet[Chap. 18]{thomee_galerkin_2006}, \eqref{eq:DFN_estimation_Pr} will be proved.
      %Finally, by the approximation properties of orthogonal projection and the duality technique \citep[Chap. 18]{thomee_galerkin_2006}, \eqref{eq:DFN_estimation_Pr} is obtained.
    \end{proof}
    \begin{lemma}\label{lemma:DFN_projection_dr_surf}
      If $w\in L^2\left(\Omega ; H_r^2\left(0, R\right)\right)$ and $b$ is independent of $r$, there exists a constant C depending on $b $, $\lambda$ and $R$  such that 
      \begin{equation}\label{eq:DFN_projection_dr_surf}
          \norm{w\qty(\cdot,R) - P_{\Delta r}w (\cdot,R)}_{L^2\qty(\Omega)} \leq C \qty(\Delta r)^{2}\|w\|_{ L^2\left(\Omega ; H_r^2\left(0, R\right)\right)}. 
      \end{equation}
    \end{lemma}
    \begin{proof}
      Given $k>0$, setting $\alpha = \sqrt{\frac{\lambda}{k}}$, $\beta = k\qty(\alpha R \qty(e^{\alpha R}+e^{-\alpha R}) - \qty(e^{\alpha R}-e^{- \alpha R})) $, it is easy to check that 
      \begin{equation*}
        G(r) = -\frac{1}{\beta} \frac{e^{\alpha r} - e^{-\alpha r}}{r} \in C^{\infty}\qty[0,R]
      \end{equation*}
      %$ G(r) = -\frac{1}{\beta} \frac{e^{\alpha r} - e^{-\alpha r}}{r} \in C^{\infty}\qty[0,R]$ 
      is a solution of the boundary value problem: 
      \begin{equation}\label{eq:DFN_projection_dr_surf_aux}
          \left\{\begin{array}{l}
          -\frac{\partial}{\partial r}\left(k r^2 \frac{\partial u}{\partial r}\right) + \lambda  u r^2=0 \,\,\,\, \mathrm{ in }\,\,\,\, \qty(0,R),  \\
          -\left. k r^2 \frac{\partial u}{\partial r}\right|_{r=0}=0, \\
          -\left. k r^2 \frac{\partial u}{\partial r}\right|_{r=R}=1.
          \end{array}\right.
      \end{equation}
      Hence for $x \in \Omega$, a.e., taking $k$ as $b(x)$ in \eqref{eq:DFN_projection_dr_surf_aux}, we have 
      \begin{displaymath}
          \qty(w - P_{\Delta r}w) \qty(x,R) =  \int_0^{R}\left(b(x) \frac{\partial\left(P_{\Delta r} w-w\right)(x,\cdot)}{\partial r} \frac{d G}{d r}+\lambda\left(P_{\Delta r} w-w\right)(x,\cdot) G\right) r^2 \dd r .
      \end{displaymath}
      Let $\operatorname{I}_{\Delta r}\colon H^1\qty(0,R) \rightarrow V^{(1)}_{\Delta r}\qty[0,R]$ be the nodal Lagrange interpolation operator and then by \eqref{eq:DFN_r_orthogonal},
      \begin{subequations} \label{eq:DFN_projection_dr_surf_x}
          \begin{align}
          \qty|\qty(w - P_{\Delta r}w) \qty(x,R)| \nonumber
          =&  \qty|\int_0^{R}\left(b \frac{\partial\left(P_{\Delta r} w-w\right)(x,\cdot)}{\partial r} \frac{d \qty(G-\operatorname{I}_{\Delta r} G)}{d r}+\lambda\left(P_{\Delta r} w-w\right)(x,\cdot) \qty(G-\operatorname{I}_{\Delta r} G)\right) r^2 \dd r | \nonumber  \\
          \leq & C\left\|w(x,\cdot)-P_{\Delta r} w(x,\cdot)\right\|_{H_r^{1}(0, R)}\|G-\operatorname{I}_{\Delta r} G\|_{H_r^{1}(0, R)} \nonumber \\
          \leq & C\left\|w(x,\cdot)-P_{\Delta r} w(x,\cdot)\right\|_{H_r^{1}(0, R)}\|G-\operatorname{I}_{\Delta r} G\|_{H^{1}(0, R)} \nonumber \\
            \leq & C(\Delta r)^2\|w(x,\cdot)\|_{ H_r^2(0, R)}.\nonumber 
          \end{align}
      \end{subequations}
      Integrating square of both sides over $\Omega$, \eqref{eq:DFN_projection_dr_surf} follows.
    \end{proof}
    
    \begin{lemma}\label{lemma:DFN_projection_Ph}
        If $w\in H^1\left(\Omega ; H_r^1\left(0, R\right)\right) \cap L^2\left(\Omega ; H_r^2\left(0, R\right)\right) $, then we have
        \begin{equation}\label{eq:DFN_estimation_Ph}
            \left\|w-P_h w\right\|_{L^2\left(\Omega; H_r^1(0, R)\right)} \leq C h\|w\|_{H^1\left(\Omega ; H_r^1\left(0, R\right)\right)} ,
        \end{equation}
    \end{lemma}
    \begin{proof}
        Since $P_h$ is the orthogonal projection, we have the best approximation property that 
        \begin{equation}\label{eq:DFN_Ph_cea_lemma}
            \norm{w-P_h w}_{L^2\left(\Omega ; H_r^1\left(0, R\right)\right)} \le \mathrm{inf}_{v_h \in V_h^\qty(0)\left(\bar \Omega\right) \otimes H_{r}^{1}\qty(0, R)} \norm{w-v_h}_{L^2\left(\Omega ; H_r^1\left(0, R\right)\right)}.
        \end{equation}
        Define the operator $P_0^x: L^2\left(\Omega ; H_r^1\left(0, R\right)\right) \rightarrow V_h^{(0)}\left(\bar\Omega\right) \otimes H_{r}^1(0, R)$,
        \begin{equation}
            P_0^x v(x , r)= \sum_{l=1}^{M} \frac{\int_{{e}_l} v(y, r) \dd y}{\left|{e}_l\right|} \chi_l(x), \quad \qty(x,r) \in \Omega\times \qty(0,R) \, \, \, \, \mathrm{a.e.},
        \end{equation} 
        where distinct $e_l \in \mathcal{T}_h$ and $\bar\Omega = \bigcup_{l=1}^{M} e_l$.  
        By the approximation theory of the $L^2$-projection operator $P_0: L^2\left(\Omega\right) \rightarrow$ $V_h^{(0)}\left(\bar\Omega\right)$, for $w \in H^1\left(\Omega ; H_r^1\left(0, R\right)\right) $, we have  
        \begin{gather}
            \left\|w-P_0^x w\right\|_{L^2\left(\Omega; L_r^2(0, R)\right)} \leq C h\|w\|_{H^1\left(\Omega ; L_r^2\left(0, R\right)\right)},\label{eq:DFN_projection_P0x_error_w} \\ 
            \left\|\pdv{w}{r} - P_0^x \pdv{w}{r}\right\|_{L^2\left(\Omega; L_r^2(0, R)\right)} \leq C h\norm{\pdv{w}{r}}_{H^1\left(\Omega ; L_r^2\left(0, R\right)\right)},\label{eq:DFN_projection_P0x_error_dw}
        \end{gather}
        % and continuity with respect to norm $L^2\left(\Omega; L_r^2(0, R)\right)$ that
        % \begin{equation}
        %     \left\|P_0^x v\right\|_{L^2\left(\Omega; L_r^2(0, R)\right)} \leq C\|v\|_{L^2\left(\Omega ; L_r^2\left(0, R\right)\right)} ,
        % \end{equation}
        % \begin{equation}
        %     \left\|P_0^x \pdv{w}{r}\right\|_{L^2\left(\Omega; L_r^2(0, R)\right)} \leq C\norm{\pdv{w}{r}}_{L^2\left(\Omega ; L_r^2\left(0, R\right)\right)},
        % \end{equation}
        where $C>0$ only depends on the shape regularity of $\mathcal{T}_h$.
        It remains to prove that operators $P_0^x $ and $\pdv{}{r}$ are commutative so that
        \begin{equation}\label{eq:DFN_projection_P0x_error_dP0w}
            \left\|\pdv{w}{r}- \pdv{P_0^x w}{r}\right\|_{L^2\left(\Omega; L_r^2(0, R)\right)} \leq C h\norm{\pdv{w}{r}}_{H^1\left(\Omega ; L_r^2\left(0, R\right)\right)}.
        \end{equation}
        % and
        % \begin{equation}
        %     \left\| \pdv{P_0^x w}{r}\right\|_{L^2\left(\Omega; L_r^2(0, R)\right)} \leq C\norm{\pdv{w}{r}}_{L^2\left(\Omega ; L_r^2\left(0, R\right)\right)},
        % \end{equation}
        If so, just take $v_h = P_0^x w$ in \eqref{eq:DFN_Ph_cea_lemma} and then \eqref{eq:DFN_estimation_Ph} is proved by \eqref{eq:DFN_projection_P0x_error_w} and \eqref{eq:DFN_projection_P0x_error_dP0w}. 

        Define $D_r\colon H_r^2(0, R) \rightarrow \mathbb{R}$, $D_r\qty(v):= \pdv{v}{ r}\ (r)$. Since $\qty|\pdv{w}{r}\ (r)| \le C\qty(r) \norm{\pdv{w}{r}}_ {H_r^1(0, R)}$ by \citet[Lem. 4]{bermejo_numerical_2021} or \citet[Lem. 2.4]{schreiber_finite_1981}, we can conclude that $D_r \in \qty[H_r^2(0, R)]'$  and thus for $w \in L^2\left(\Omega ; H_r^2\left(0, R\right)\right)$,
        \begin{equation}
            \begin{aligned}
                P_0^x \pdv{w}{r}\ (x , r) & = \sum_{l=1}^{M} \frac{\int_{{e}_1} \pdv{w}{r}\ (y, r) \dd y}{\left|{e}_l\right|} \chi_l(x)  
                                        = \sum_{l=1}^{M} \frac{\int_{{e}_1} \qty(D_r, w(y,\cdot)) \dd y}{\left|{e}_l\right|} \chi_l(x) \\ 
                                        &  = \sum_{l=1}^{M} \frac{ \qty(D_r, \int_{{e}_1} w(y,\cdot) \dd y ) }{\left|{e}_l\right|} \chi_l(x)
                                        = \qty(D_r, \sum_{l=1}^{M} \frac{ \int_{{e}_1} w(y,\cdot) \dd y  }{\left|{e}_l\right|} \chi_l(x)) \\ & =  \pdv{P_0^x w}{r}\ (x , r)\qcomma{} \qty(x,r) \in \Omega\times \qty(0,R) \, \, \, \, \mathrm{a.e.},
            \end{aligned}
        \end{equation}
        i.e., $P_0^x $ and $\pdv{}{r}$ are commutative, which completes the proof.
    \end{proof}

    \begin{theorem}\label{lemma:DFN_projection_hdr}
        If $w\in H^1\left(\Omega ; H_r^1\left(0, R\right)\right) \cap L^2\left(\Omega ; H_r^2\left(0, R\right)\right) $, then for $q=0,1$,
        \begin{equation*}\label{eq:DFN_projection_hdr_error}
            \norm{w-P_{h\Delta r}w}_{L^2\left(\Omega ; H_r^q\left(0, R\right)\right)} \le C \qty(h\|w\|_{H^1\left(\Omega ; H_r^1\left(0, R\right)\right)} + \qty(\Delta r)^{2-q} \|w\|_{L^2\left(\Omega ; H_r^2\left(0, R\right)\right)}).
        \end{equation*}
        % If $w$ and $b$ are further dependent of the time parameter $t$, $\dv{w}{t}\in H^1\left(\Omega ; H_r^1\left(0, R\right)\right) \cap L^2\left(\Omega ; H_r^ 2\left(0, R\right)\right)$, $\dv{b}{t}\in L^\infty\qty(\Omega)$, then we have
        % \begin{multline}\label{eq:DFN_projection_hdr_dt_error}
        %     \norm{\dv{\qty(w-P_{h\Delta r}w)}{t}}_{L^2\left(\Omega ; H_r^1\left(0, R\right)\right)} \le Ch \qty(\|w\|_{H^1\left(\Omega ; H_r^1\left(0, R\right)\right)} + \norm{\pdv{w}{t}}_{H^1\left(\Omega ; H_r^1\left(0, R\right)\right)} ) \\ 
        %     + C\Delta r  \qty(\|w\|_{L^2\left(\Omega ; H_r^2\left(0, R\right)\right)} + \norm{\pdv{w}{t}}_{L^2\left(\Omega ; H_r^2\left(0, R\right)\right)}).
        % \end{multline}
    \end{theorem}
    \begin{proof}
        % Since the  bilinear type \eqref{eq:DFN_r_orthogonal_a} is symmetric, coercive and continuous, it is easy to get
        % \begin{equation}\label{eq:DFN_Phr_cea_lemma}
        % \norm{w-P_{h\Delta r} w}_{L^2\left(\Omega ; H_r^1\left(0, R\right)\right)} \le \mathrm{inf}_{ v_h \in V_h^\qty(0)\left(\bar \Omega\right) \otimes V_{\Delta r}^\qty(1)\qty[0, R]} \norm{w-v_h}_{L ^2\left(\Omega ; H_r^1\left(0, R\right)\right)}.
        % \end{equation}
        % Taking $v_{h\Delta r} = P_h P_{\Delta r} w \in V_h^\qty(0)\left(\bar \Omega\right) \otimes V_{\Delta r}^\qty(1)\qty[0, R]$,    
        By means of Lemmas~\ref{lemma:DFN_projection_dr} and \ref{lemma:DFN_projection_Ph}, it readily follows for $q = 0, 1$,  
        \begin{subequations}
            \begin{align}
                \norm{w-P_{h\Delta r} w}_{L^2\left(\Omega ; H_r^q\left(0, R\right)\right)}  
                = &\norm{w -  P_{\Delta r} P_h w}_{L^2\left(\Omega ; H_r^q\left(0, R\right)\right)} \nonumber \\
                 \le &\norm{w - P_{\Delta r} w}_{L^2\left(\Omega ; H_r^q\left(0, R\right)\right)} + \norm{P_{\Delta r} \qty(I - P_{h}) w}_{L^2\left(\Omega ; H_r^1\left(0, R\right)\right)} \nonumber\\
                 \le &\norm{w - P_{\Delta r} w}_{L^2\left(\Omega ; H_r^q\left(0, R\right)\right)} + \norm{w - P_{h} w}_{L^2\left(\Omega ; H_r^1\left(0, R\right)\right)} \nonumber\\
                 \le &C \qty( \qty(\Delta r)^{2-q} \norm{w}_ {L^2\left(\Omega ; H_r^2\left(0, R\right)\right)} + h \norm{w}_{H^1\left(\Omega ; H_r^1\left(0, R\right)\right)} ).\nonumber
            \end{align}
        \end{subequations}

    \end{proof}
    
    \begin{theorem}\label{lemma:DFN_projection_hdr_surf}
        If $w\in H^1\left(\Omega ; H_r^1\left(0, R\right)\right) \cap L^2\left(\Omega ; H_r^2\left(0, R\right)\right) $, then we have
        \begin{equation}\label{eq:DFN_projection_hdr_surf_error}
            \norm{\qty(w-P_{h\Delta r}w)\qty(\cdot,R)}_{L^2\left(\Omega\right)} \le C \qty(h\|w\|_{H^1\left(\Omega ; H_r^1\left(0, R\right)\right)} + \qty(\Delta r)^2 \|w\|_{L^2\left(\Omega ; H_r^2\left(0, R\right)\right)}).
        \end{equation}
    \end{theorem}
    \begin{proof}
        By means of Lemmas~\ref{lemma:DFN_projection_dr_surf} and \ref{lemma:DFN_projection_hdr}, it readily follows 
        \begin{equation*}
            \begin{aligned}
                \norm{\qty(w-P_{h\Delta r}w)\qty(\cdot,R)}_{L^2\left(\Omega\right)} 
                \le & \norm{\qty(w-P_{\Delta r}w)\qty(\cdot,R)}_{L^2\left(\Omega\right)} + \norm{\qty(P_{\Delta r}\qty(I-P_{h})w)\qty(\cdot,R)}_{L^2\left(\Omega\right)} \\ 
                \le & C \qty(\Delta r)^2 \norm{w}_{L^2\left(\Omega; H^2_r\qty(0,R)\right)} + \norm{P_{\Delta r}\qty(w-P_{h}w)}_{L^2\left(\Omega; H^1_r\qty(0,R)\right)} \\ 
                \le & C \qty(\Delta r)^2 \norm{w}_{L^2\left(\Omega; H^2_r\qty(0,R)\right)} + \norm{w-P_{h}w}_{L^2\left(\Omega; H^1_r\qty(0,R)\right)} \\ 
                \le & C \qty(h\|w\|_{H^1\left(\Omega ; H_r^1\left(0, R\right)\right)} + \qty(\Delta r)^2 \|w\|_{L^2\left(\Omega ; H_r^2\left(0, R\right)\right)}).
            \end{aligned}
        \end{equation*} 
        The second inequality is due to Proposition~\ref{prop:DFN_radial_surface}.
    \end{proof}

  \subsection{Error estimation for $c_2$}
    Let $P_{h\Delta r,m}$, $m\in \set{\mathrm{n},\mathrm{p}}$, denote the operator defined by \eqref{eq:DFN_c2_projection_hr} with $ b = k_2$, $\lambda = 1$, $\Omega = \Omega_{m}$ and $R=R_m$. Then, we define
    \begin{equation*}
        P_{h\Delta r} v (x,r):=\begin{cases}
        P_{h\Delta r,\mathrm{n}} v (x,r)& \text{if $(x,r) \in \Omega_\mathrm{n}\times\qty(0,R_\mathrm{n})$},\\
        P_{h\Delta r,\mathrm{p}} v (x,r)& \text{if $(x,r) \in \Omega_\mathrm{p}\times\qty(0,R_\mathrm{p})$},
        \end{cases}
    \end{equation*}
    for $v\in L^2\left(\Omega_2 ; H_r^1\left(0, R_\mathrm{s}(\cdot)\right)\right)$.
    Thanks to the newly introduced projection operator $P_{h\Delta r}$, the finite element approximation error can be decomposed in a standard way as follows:
    \begin{equation}\label{eq:DFN_fem_c2_err_decomp}
        c_2-c_{2h\Delta r}=\left(c_2-P_{h\Delta r} c_2\right)+\left(P_{h\Delta r} c_2 -c_{2h \Delta r}\right)=:\rho_{c_2} +\theta_{c_2}.
    \end{equation}
    % \begin{gather}
    %     \rho_{c_2}(x , r, t):= c_2(x , r, t) -P_{h\Delta r} c_2(x , r, t) ,\label{eq:DFN_fem_c2_err_decomp_rho2} \\ 
    %     \theta_{c_2}(x , r, t):= P_{h\Delta r} c_2(x , r, t) -c_{2h \Delta r}(x , r, t).\label{eq:DFN_fem_c2_err_decomp_theta2}
    % \end{gather}
    The estimate for $\rho_{c_2}$ has been established; we now proceed to provide an estimate for $\theta_{c_2}$.

    \begin{lemma}\label{lemma:DFN_fem_estimation_theta_c2}
        
      Assuming that $c_2 \in H^1\left(0, T ; H^1\left(\Omega_2 ; H_r^1(0, R_\mathrm{s}\qty(\cdot))\right) \cap L^2\left(\Omega_2 ; H_r^2\qty(0, R_\mathrm{s}\qty(\cdot))\right)\right)$ and that Assumptions~\ref{asp:nonlinear_function} and \ref{asp:DFN_prior_bound}-\ref{asp:DFN_fem_prior_bound} hold, there exist an arbitrarily small number $\epsilon > 0$ and a positive constant $C\qty(\epsilon) $ that do not depend on $h$ and $\Delta r$, such that for $t \,\,\mathrm{ a.e.} \,\,   \mathrm{in }\,\, \qty[0,T]$, 
      \begin{multline}\label{eq:DFN_fem_estimation_theta_c2}
        \frac{d}{d t}  \left\|\theta_{c_2}(t)\right\|_{L^2\left(\Omega_2; L_r^2\left(0, R_\mathrm{s}(\cdot)\right)\right)}^2+\underline{k_2}\left\|\frac{\partial \theta_{c_2}}{\partial r}(t)\right\|_{L^2\left(\Omega_2; L_r^2\left(0, R_\mathrm{s}(\cdot)\right)\right)}^2 \\
        \begin{aligned}
          \leq &\epsilon\left(\norm{\qty(c_1 - c_{1h})(t)}_{L^2\qty(\Omega)}^2 + \norm{\qty(\phi_1-\phi_{1h})(t)}_{L^2\qty(\Omega)}^2 + \norm{\qty(\phi_2-\phi_{2h})(t)}^2_{L^2\qty(\Omega_2)}\right) \\
          &+ C\qty(\epsilon) \qty(\Delta r)^4\left(\|c_2(t)\|_{L^2\left(\Omega_2; H_r^2\left(0, R_\mathrm{s}(\cdot)\right)\right)}^2+\left\|\frac{\partial c_2(t)}{\partial t}\right\|_{L^2\left(\Omega_2; H_r^2\left(0, R_\mathrm{s}(\cdot)\right)\right)}^2\right) \\ 
          & +C\qty(\epsilon) h^2\left(\|c_2(t)\|_{H^1\left(\Omega_2; H_r^1\left(0, R_\mathrm{s}(\cdot)\right)\right)}^2+\left\|\frac{\partial c_2(t)}{\partial t}\right\|_{H^1\left(\Omega_2; H_r^1\left(0, R_\mathrm{s}(\cdot)\right)\right)}^2\right) \\ 
          & +C\qty(\epsilon)\left\|\theta_{c_2}(t)\right\|^2_{L^2\left(\Omega_2; L_r^2\left(0, R_\mathrm{s}(\cdot)\right)\right)} .
        \end{aligned}
      \end{multline}
    \end{lemma}

    \begin{proof}
      Note that $c_{2 h \Delta r}(t) = c_{2}(t) - \rho_{c_2}(t) - \theta_{c_2}(t)$. Substitute it into \eqref{eq:DFN_fem_c2} and by means of \eqref{eq:DFN_weak_c2} and \eqref{eq:DFN_c2_projection_hr}, we have
      \begin{equation*}
          \begin{aligned}
          & \int_{\Omega_2} \int_0^{R_\mathrm{s}(x)} \dv{ \theta_{c_2}}{ t} \qty(t) v_{h \Delta r} r^2 \dd r \dd x+\int_{\Omega_2} \int_0^{R_\mathrm{s}(x)} k_{2} \frac{\partial \theta_{c_2}(t)}{\partial r} \frac{\partial v_{ h \Delta r}}{\partial r} r^2 \dd r \dd x \\
          =&\int_{\Omega_2} \int_0^{R_\mathrm{s}(x)}  \rho_{c_2}\qty(t) v_{h \Delta r} r^2 \dd r \dd x -\int_{\Omega_2} \int_0^{R_\mathrm{s}(x)} \dv{\rho_{c_2}}{t}\qty(t) v_{h \Delta r} r^2 \dd r \dd x \\
          &  + \int_{\Omega_2} \frac{R_\mathrm{s}^2(x)}{F}\left(J_h(t,x)-J(t,x)\right) v_{h \Delta r}\left(x, R_\mathrm{s}(x)\right) \dd x.
          \end{aligned}
      \end{equation*}

      From now on, we omit $\qty(t)$ for brevity. Taking $v_{h \Delta r}=\theta_{c_2}\in V_{h\Delta r}(\bar \Omega_{2r})$, H\"older's inequality yields
      \begin{multline}\label{eq:DFN_theta2_estimation_raw}
          \frac{1}{2} \frac{d}{d t}\left\|\theta_{c_2}\right\|_{L^2\left(\Omega_2 ; L_r^2\left(0, R_\mathrm{s}(\cdot)\right)\right)}^2+\underline{k_2}\left\|\frac{\partial \theta_{c_2}}{\partial r}\right\|_{L^2\left(\Omega_2 ; L_r^2\left(0, R_\mathrm{s}(\cdot)\right)\right)}^2  
          \leq  C\left\|J_h-J\right\|_{L^2\left(\Omega_2 \right)} \| \theta_{c_2}\qty(\cdot, R_\mathrm{s}(\cdot)) \|_{L^2\left(\Omega_2 \right)} \\
               +\left\|\frac{\partial \rho_{c_2}}{\partial t}\right\|_{L^2\left(\Omega_2 ; L_r^2\left(0, R_\mathrm{s}(\cdot)\right)\right)}\left\|\theta_{c_2}\right\|_{L^2\left(\Omega_2 ; L_r^2\left(0, R_\mathrm{s}(\cdot)\right)\right)} 
               +\left\|\rho_{c_2}\right\|_{L^2\left(\Omega_2 ; L_r^2\left(0, R_\mathrm{s}(\cdot)\right)\right)}\left\|\theta_{c_2}\right\|_{L^2\left(\Omega_2 ; L_r^2\left(0, R_\mathrm{s}(\cdot)\right)\right)} .
      \end{multline}
      Since Assumptions \ref{asp:nonlinear_function}, \ref{asp:DFN_prior_bound} and \ref{asp:DFN_fem_prior_bound}  are satisfied, the Lipschitz continuity (cf. \citet[Lem. 6]{bermejo_numerical_2021} and \citet[Lem. 4.1]{kroener_mathematical_2016}) implies that
      \begin{equation} \label{eq:DFN_estimation_J}
      \norm{J-J_h}_{L^2\qty(\Omega_2)} \le C \left(\norm{c_1-c_{1h}}_{L^2\qty(\Omega)} + \norm{\bar{c}_2-\bar{c}_{2h}}_{L^2\qty(\Omega_2)} + \norm{\phi_1-\phi_{1h}}_{L^2\qty(\Omega)} + \norm{\phi_2-\phi_{2h}} _{L^2\qty(\Omega_2)} \right).
      \end{equation}
      %Hence, it remains to estimate $\norm{\bar{c}_2-\bar{c}_{2h}}_{L^2\qty(\Omega_2)}$ and $\| \theta_{c_2}\qty(\cdot, R_\mathrm{s}(\cdot)) \|_{L^2\left(\Omega_2 \right)}$.
      Notice that from \eqref{eq:DFN_fem_c2_err_decomp}, 
      \begin{equation}\label{eq:DFN_fem_c2_surf_triangle}
        \norm{\bar{c}_2-\bar{c}_{2h}}_{L^2\qty(\Omega_2)}  \le \norm{\rho_2(\cdot,R_\mathrm{s}(\cdot))}_{L^2\qty(\Omega_2)} + \left\|\theta_{c_2}\left(\cdot, R_\mathrm{s}(\cdot)\right)\right\|_{L^2\left(\Omega_2\right)}.
      \end{equation} 
    %   \begin{displaymath}
    %       \qty|\bar{c}_2(x)-\bar{c}_{2h}(x)| \le \qty|\rho_2(x,R_\mathrm{s}(x))| + \qty|\theta_2(x,R_\mathrm{s}(x))|.
    %   \end{displaymath}
      According to Lemma~\ref{lemma:DFN_projection_hdr_surf},
      \begin{equation}\label{eq:DFN_fem_theta2_surf_L2}
          \norm{\rho_2(\cdot,R_\mathrm{s}(\cdot))}_{L^2\qty(\Omega_2)} \le C \qty(h\|c_2\|_{H^1\left(\Omega_2 ; H_r^1\left(0, R_\mathrm{s}(\cdot)\right)\right)} + \qty(\Delta r)^2 \|c_2\|_{L^2\left(\Omega_2 ; H_r^2\left(0, R_\mathrm{s}(\cdot)\right)\right)}),
      \end{equation}
      while by Proposition~\ref{prop:DFN_radial_surface}, there exists an arbitrarily small $\tilde{\epsilon} > 0$, such that
    %   \begin{displaymath}  
    %       \left|\theta_{c_2}\left(x, R_\mathrm{s}(x)\right)\right| 
    %       \leq \tilde \epsilon\left\|\frac{\partial \theta_{c_2}}{\partial r}\ (x)\right\|_{L_r^2\left(0, R_\mathrm{s}(x)\right)}  + 
    %       C(\tilde \epsilon)\left\|\theta_{c_2}(x)\right\|_{L_r^2\left(0, R_\mathrm{s}(x)\right)},
    %   \end{displaymath}
    %   which implies
      \begin{equation}\label{eq:DFN_theta2_surface_estimation} 
          \left\|\theta_{c_2}\left(\cdot, R_\mathrm{s}(\cdot)\right)\right\|_{L^2\left(\Omega_2\right)} \leq \tilde  \epsilon\left\|\frac{\partial \theta_{c_2}}{\partial r}\right\|_{L^2\left(\Omega_2 ; L_r^2\left(0, R_\mathrm{s}(\cdot)\right)\right)}
          +C(\tilde \epsilon)\left\|\theta_{c_2}\right\|_{L^2\left(\Omega_2 ; L_r^2\left(0, R_\mathrm{s}(\cdot)\right)\right)}.
      \end{equation}
      Then it follows
      \begin{multline}\label{eq:DFN_c2_surface_estimation} 
          \norm{\bar{c}_2-\bar{c}_{2h}}_{L^2\qty(\Omega_2)}  \le C \qty(h\|c_2\|_{H^1\left(\Omega_2 ; H_r^1\left(0, R_\mathrm{s}(\cdot)\right)\right)} + \qty(\Delta r)^2 \|c_2\|_{L^2\left(\Omega ; H_r^2\left(0, R_\mathrm{s}(\cdot)\right)\right)})  \\ +   \qty(\tilde \epsilon\norm{\pdv{\theta_{c_2}}{r}}_{L^2\qty(\Omega_2;L_r^2\qty(0,R_\mathrm{s}(\cdot)))} +  C(\tilde \epsilon)\norm{\theta_{c_2}}_{L^2\qty(\Omega_2;L_r^2\qty(0,R_\mathrm{s}(\cdot)))}).
      \end{multline}
      Young's inequality, together with \eqref{eq:DFN_estimation_J} and \eqref{eq:DFN_fem_c2_surf_triangle}, yields that for arbitrarily small number $\epsilon > 0$, 
      \begin{multline}
        C\left\|J_h-J\right\|_{L^2\left(\Omega_2 \right)} \| \theta_{c_2}\qty(\cdot, R_\mathrm{s}(\cdot)) \|_{L^2\left(\Omega_2 \right)} 
        \le  C\qty(\epsilon)\left\|\theta_{c_2}\left(\cdot, R_\mathrm{s}(\cdot)\right)\right\|_{L^2\left(\Omega_2\right)}^2  + \epsilon\norm{\rho_2(\cdot,R_\mathrm{s}(\cdot))}^2_{L^2\qty(\Omega_2)} \\
        +  \epsilon\left(\norm{\qty(c_1 - c_{1h})(t)}_{L^2\qty(\Omega)}^2+ \norm{\qty(\phi_1-\phi_{1h})(t)}_{L^2\qty(\Omega)}^2 + \norm{\qty(\phi_2-\phi_{2h})(t)}^2_{L^2\qty(\Omega_2)}\right). 
      \end{multline}
      We can further select $\tilde \epsilon$ sufficiently small in \eqref{eq:DFN_theta2_surface_estimation}, such that   
      \begin{equation*}
        C\qty(\epsilon)\left\|\theta_{c_2}\left(\cdot, R_\mathrm{s}(\cdot)\right)\right\|_{L^2\left(\Omega_2\right)}^2
          \le \frac{\underline{k_2}}{2}  \left\|\frac{\partial \theta_{c_2}}{\partial r}\right\|_{L^2\left(\Omega_2; L_r^2\left(0, R_\mathrm{s}(\cdot)\right)\right)}^2 
          +  C\qty(\epsilon)\left\|\theta_{c_2}\right\|^2_{L^2\left(\Omega_2; L_r^2\left(0, R_\mathrm{s}(\cdot)\right)\right)}.
      \end{equation*}
      By applying the Cauchy-Schwarz inequality to the remaining terms on the right-hand side of \eqref{eq:DFN_theta2_estimation_raw}, and using Theorem~\ref{lemma:DFN_projection_hdr} together with \eqref{eq:DFN_fem_theta2_surf_L2}, we derive the final result.
    %   \begin{equation*}
    %     \begin{aligned} 
    %       &\frac{d}{d t}  \left\|\theta_{c_2}\right\|_{L^2\left(\Omega_2; L_r^2\left(0, R_\mathrm{s}(\cdot)\right)\right)}^2+\underline{k_2}\left\|\frac{\partial \theta_{c_2}}{\partial r}\right\|_{L^2\left(\Omega_2; L_r^2\left(0, R_\mathrm{s}(\cdot)\right)\right)}^2 \\
    %       \leq 
    %       & C\qty(\epsilon) \qty(\Delta r)^4\left(\|c_2\|_{L^2\left(\Omega_2; H_r^2\left(0, R_\mathrm{s}(\cdot)\right)\right)}^2+\left\|\frac{\partial c_2}{\partial t}\right\|_{L^2\left(\Omega_2; H_r^2\left(0, R_\mathrm{s}(\cdot)\right)\right)}^2\right)  \\
    %       & +C\qty(\epsilon) h^2\left(\|c_2\|_{H^1\left(\Omega_2; H_r^1\left(0, R_\mathrm{s}(\cdot)\right)\right)}^2+\left\|\frac{\partial c_2}{\partial t}\right\|_{H^1\left(\Omega_2; H_r^1\left(0, R_\mathrm{s}(\cdot)\right)\right)}^2\right) \\ 
    %       & +\epsilon\left(\norm{c_1 - c_{1h}}_{L^2\qty(\Omega)}^2 + \norm{\phi_1-\phi_{1h}}_{L^2\qty(\Omega)}^2 + \norm{\phi_2-\phi_{2h}}^2_{L^2\qty(\Omega_2)}\right) \\
    %       &+C\qty(\epsilon)\left\|\theta_{c_2}\right\|^2_{L^2\left(\Omega_2; L_r^2\left(0, R_\mathrm{s}(\cdot)\right)\right)}.
    %     \end{aligned}
    %   \end{equation*}
    \end{proof}

  \subsection{Error estimation for the semi-discrete problems}
    To this end, we still require error estimates for $\phi_1$, $\phi_2$ and ${c_1}$, which will then be combined to address the fully coupled problems.

    We first define the projection operator $P_{h}^3$ from $H^1(\Omega)$ to $V_h^{(1)}(\bar \Omega)$ such that $\forall u \in H^1(\Omega)$,
    \begin{equation}\label{eq:DFN_projection_c1}
        \int_{\Omega} k_1 \nabla\left(u-P_{h}^3 u\right) \cdot \nabla \varphi_h \, \dd x+\int_{\Omega}  \left(u-P_{h}^3 u\right) \varphi_h \, \dd x =0, \quad \forall \varphi_h \in V_h^{(1)}(\Omega).
    \end{equation}
    For $u \in H_\mathrm{pw}^2\qty(\Omega_1)$, it holds $\norm{u-P_{h}^3 u}_{H^1\qty(\Omega_1)}  \le Ch\norm{u}_{H_\mathrm{pw}^2\qty(\Omega_1)}$ \citep{xu1982estimate}.
    Then we decompose the finite element approximation error as
    \begin{equation}\label{app_eq:DFN_fem_c1_err_decomp}
        c_1\qty(t)-c_{1h}\qty(t)=\left(c_1\qty(t)-P_{h}^3 c_1\qty(t)\right)+\left(P_{h}^3 c_1\qty(t)-c_{1h}\qty(t)\right)=:\rho_{c_1}\qty(t) +\theta_{c_1}\qty(t),
    \end{equation}
    and it readily follows that
    \begin{gather}
        \label{eq:DFN_fem_c1_projection_error}
            \norm{\rho_{c_1}(t)}_{H^1\qty(\Omega)}  \le Ch\norm{c_1(t)}_{H_\mathrm{pw}^2\qty(\Omega_1)}, \\ 
        \label{eq:DFN_fem_c1t_projection_error}
            \norm{\dv{\rho_{c_1}}{t}\qty(t)}_{H^1\qty(\Omega)} \le Ch\norm{\dv{c_1}{t} \qty(t)}_{H_\mathrm{pw}^2\qty(\Omega_1)}.
    \end{gather}

    The remaining error estimations for $\phi_1$, $\phi_2$ and $\theta_{c_1}$ follow standard procedures.  Here, we state the results without proof, which can be found in detail in Sections~\ref{app_sec:DFN_err_phi} and \ref{app_sec:DFN_err_theta_c1}. Since $\varepsilon_1$ is a piecewise constant positive function and independent of $t$, the norms $\norm{\varepsilon_1^{1/2}u}_{L^2\qty(\Omega)}$ and $\norm{u}_{L^2\qty(\Omega)}$ are equivalent. Henceforth, we assume $\varepsilon_1 \equiv 1$.
    \begin{lemma}\label{lemma:DFN_fem_estimation_Phi}
        Assume that $\phi_1\in L^2\qty(0,T;H^2_\mathrm{pw}(\Omega_1)) $, $\phi_2\in  L^2\qty(0,T;H^2_\mathrm{pw}(\Omega_2)) $ and that Assumptions~\ref{asp:nonlinear_function} and \ref{asp:DFN_prior_bound}-\ref{asp:DFN_fem_prior_bound} hold. There is a constant $C$ that does not depend on $h$ and $\Delta r$, such that for $t\in \qty[0,T]$ a.e.,
        \begin{multline}\label{eq:DFN_fem_estimation_Phi}
            \norm{\phi_1(t)-\phi_{1h}(t)}_{H^1\qty(\Omega)}^2 + \norm{\phi_2(t)-\phi_{2h}(t)}^2_{H^1\qty(\Omega_2)}  \\ \le Ch^2\qty(\norm{\phi_1(t)}_{H^2_\mathrm{pw}(\Omega_1)}^2 + \norm{\phi_2(t)}_{H^2_\mathrm{pw}(\Omega_2)}^2)   
            + C\qty( \norm{c_1(t) - c_{1h}(t)}_{H^1(\Omega)}^2 + \norm{\bar c_2(t) - \bar c_{2h}(t)}_{L^2(\Omega_2)}^2 ).
        \end{multline}
    \end{lemma}

    \begin{remark}
        Note the dependence on the $H^1$-norm on the right-hand side of \eqref{eq:DFN_fem_estimation_Phi}, which results from not performing the change of variables. This dependence motivates the introduction of the undetermined parameter $\epsilon$ in Lemma~\ref{lemma:DFN_fem_estimation_theta_c2} and subsequently in Lemma~\ref{lemma:DFN_fem_estimation_theta_c1}.
    \end{remark}

    \begin{lemma}\label{lemma:DFN_fem_estimation_theta_c1}
        Assume that $c_1 \in H^1\qty(0,T; H^2_\mathrm{pw}(\Omega_1))$ and that Assumptions~\ref{asp:nonlinear_function} and \ref{asp:DFN_prior_bound}-\ref{asp:DFN_fem_prior_bound} hold. There exist an arbitrarily small number $\epsilon > 0$ and a positive constant $C\qty(\epsilon) $ that do not depend on $h$ and $\Delta r$, such that for $t\in \qty[0,T]$ a.e., 
        \begin{multline}\label{eq:DFN_fem_estimation_theta_c1}
            \dv{}{t}\norm{\theta_{c_1}(t)}_{L^2\qty(\Omega_1)}^2 + \underline{k_1} \norm{\nabla \theta_{c_1}(t)}_{L^2\qty(\Omega_1)}^2 \\ 
            \le C\qty(\epsilon)h^2 \qty(\norm{c_1(t)}_{H_\mathrm{pw}^2\qty(\Omega_1)}^2 + \norm{\pdv{c_1}{t} \qty(t)}_{H_\mathrm{pw}^2\qty(\Omega_1)}^2) + C\qty(\epsilon) \norm{\theta_{c_1}(t)}_{L^2\qty(\Omega)}^2  \\ 
            + \epsilon\qty(\norm{\bar c_2(t) - \bar c_{2h}(t)}_{L^2\qty(\Omega_2)}^2 + \norm{\phi_1(t)-\phi_{1h}(t)}_{L^2\qty(\Omega)}^2 + \norm{\phi_2(t)-\phi_{2h}(t)}^2_{L^2\qty(\Omega_2)}  ).
        \end{multline}
    \end{lemma}

    Based on Lemmas~\ref{lemma:DFN_fem_estimation_theta_c2}, \ref{lemma:DFN_fem_estimation_Phi} and \ref{lemma:DFN_fem_estimation_theta_c1}, we arrive at the main result of this paper, presented in the following theorem.

    \begin{theorem}\label{thm:DFN_fem_semi_estimation} 
        If Assumptions~\ref{asp:IB_condition}-\ref{asp:nonlinear_function} and \ref{asp:DFN_fem_IC}-\ref{asp:DFN_fem_prior_bound} hold, 
      %If Assumptions~\ref{asp:IB_condition}, \ref{asp:nonlinear_function}, \ref{asp:DFN_fem_IC}, \ref{asp:DFN_fem_regularity}, \ref{asp:DFN_prior_bound} and \ref{asp:DFN_fem_prior_bound} hold, 
      there exists a constant $C$ that does not depend on $h$ and $\Delta r$, such that for $t \in \qty[0,T]$ a.e.,
      \begin{multline*}
          \norm{\phi_1 - \phi_{1h}}_{L^2\qty(0,t;H^1\qty(\Omega))} + \norm{\phi_2 - \phi_{2h}}_{L^2\qty(0,t;H^1\qty(\Omega_2))}  +\norm{c_1 - c_{1h}}_{L^2\qty(0,t;H^1\qty(\Omega))} +\norm{\bar c_{2} - \bar c_{2h}}_{L^2\qty(0,t;L^2\qty(\Omega_2))} \\ \le C\qty(h + \qty(\Delta r)^2)  + \norm{c_{10}-c_{1h}^0}_{L^2\qty(\Omega)} + \norm{c_{20}-c_{2h\Delta r}^0}_{L^2\qty(\Omega_2 ; L_r^2\qty(0, R_\mathrm{s}\qty(\cdot)))},  
      \end{multline*}
      \begin{multline*}
          \left\|c_2 - c_{2h\Delta r}\right\|_{L^2\qty(0,t;L^2\left(\Omega_2; H_r^q\left(0, R_\mathrm{s}(\cdot)\right)\right))}  \\ \le C\qty(h +\qty(\Delta r)^\qty(2-q))
          + \norm{c_{10}-c_{1h}^0}_{L^2\qty(\Omega)} + \norm{c_{20}-c_{2h\Delta r}^0}_{L^2\qty(\Omega_2 ; L_r^2\qty(0, R_\mathrm{s}\qty(\cdot)))} \qcomma{} q=0,1. 
      \end{multline*}  
    \end{theorem}

    \begin{proof}
        By combining Lemma~\ref{lemma:DFN_fem_estimation_theta_c1} with Lemma~\ref{lemma:DFN_fem_estimation_Phi}, along with \eqref{eq:DFN_c2_surface_estimation} and \eqref{app_eq:DFN_fem_c1_err_decomp}-\eqref{eq:DFN_fem_c1t_projection_error}, there is a constant $\epsilon_{c_1}>0$ to be determined later, such that
      \begin{multline}\label{eq:DFN_fem_estimation_theta_c1_expand}
          \dv{}{t}\norm{\theta_{c_1}(t)}_{L^2\qty(\Omega)}^2 + \underline{k_1} \norm{\nabla \theta_{c_1}}_{L^2\qty(\Omega)}^2 \\ 
          \begin{aligned}
              \le 
              & C\qty(\epsilon_{c_1}) h^2 \qty(\norm{c_1(t)}_{H_\mathrm{pw}^2\qty(\Omega_1)}^2 + \norm{\pdv{c_1}{t} \qty(t)}_{H_\mathrm{pw}^2\qty(\Omega_1)}^2+  \norm{c_2\qty(t)}^2_{H^1\qty(\Omega_2 ; H_r^1(0, R_\mathrm{s}\qty(\cdot)))})  \\
              & + C\qty(\epsilon_{c_1}) h^2\qty( \norm{\phi_1(t)}_{H^2_\mathrm{pw}(\Omega_1)}^2 + \norm{\phi_2(t)}_{H^2_\mathrm{pw}(\Omega_2)}^2) \\ 
              & + C\qty(\Delta r)^4 \norm{c_2\qty(t)}^2_{L^2\qty(\Omega_2 ; H_r^2(0, R_\mathrm{s}\qty(\cdot)))} \\  
              & + C\qty(\epsilon_{c_1}) \qty( \norm{\theta_{c_1}\qty(t)}_{L^2\qty(\Omega)}^2 +  \norm{\theta_{c_2}\qty(t)}_{L^2\qty(\Omega_2 ; L_r^2(0, R_\mathrm{s}\qty(\cdot)))}^2)  \\ 
              & + \epsilon_{c_1} \qty( \norm{\nabla \theta_{c_1}\qty(t)}^2_{L^2\qty(\Omega)} + \norm{\pdv{\theta_{c_2}}{r} \qty(t)}_{L^2\qty(\Omega_2;L_r^2\qty(0,R_\mathrm{s}(\cdot)))}) .
          \end{aligned}
      \end{multline}
      Again, by combining the error estimates in Lemma~\ref{lemma:DFN_fem_estimation_theta_c2} with Lemma~\ref{lemma:DFN_fem_estimation_Phi}, along with \eqref{eq:DFN_c2_surface_estimation}, \eqref{app_eq:DFN_fem_c1_err_decomp} and \eqref{eq:DFN_fem_c1_projection_error}, there is also a constant $\epsilon_{c_2}>0$ to be determined later, such that
      \begin{multline}\label{eq:DFN_fem_estimation_theta_c2_expand}
          \frac{d}{d t}  \left\|\theta_{c_2}(t)\right\|_{L^2\left(\Omega_2; L_r^2\left(0, R_\mathrm{s}(\cdot)\right)\right)}^2+\underline{k_2}\left\|\frac{\partial \theta_{c_2}(t)}{\partial r}\right\|_{L^2\left(\Omega_2; L_r^2\left(0, R_\mathrm{s}(\cdot)\right)\right)}^2 \\
          \begin{aligned}
              \leq 
              &C\qty(\epsilon_{c_2}) \qty(\Delta r)^4\left(\|c_2(t)\|_{L^2\left(\Omega_2; H_r^2\left(0, R_\mathrm{s}(\cdot)\right)\right)}^2+\left\|\frac{\partial c_2}{\partial t}\qty(t)\right\|_{L^2\left(\Omega_2; H_r^2\left(0, R_\mathrm{s}(\cdot)\right)\right)}^2\right) \\
              & + C\qty(\epsilon_{c_2}) h^2\left(  \|c_2(t)\|_{H^1\left(\Omega_2; H_r^1\left(0, R_\mathrm{s}(\cdot)\right)\right)}^2+\left\|\frac{\partial c_2}{\partial t}\qty(t)\right\|_{H^1\left(\Omega_2; H_r^1\left(0, R_\mathrm{s}(\cdot)\right)\right)}^2\right) \\ 
              &+ C\qty(\epsilon_{c_2}) h^2\qty(\norm{{c_1}\qty(t)}_{H_\mathrm{pw}^2\qty(\Omega_1)}^2 + \norm{\phi_1(t)}_{H^2_\mathrm{pw}(\Omega_1)}^2 + \norm{\phi_2(t)}_{H^2_\mathrm{pw}(\Omega_2)}^2) \\
              & + C\qty(\epsilon_{c_2})\qty(\norm{\theta_{c_1}\qty(t)}_{L^2\qty(\Omega)}^2 + \left\|\theta_{c_2}(t)\right\|^2_{L^2\left(\Omega_2; L_r^2\left(0, R_\mathrm{s}(\cdot)\right)\right)}) \\
              &+ \epsilon_{c_2}\left( \norm{\nabla \theta_{c_1}\qty(t)}^2_{L^2\qty(\Omega)} + \norm{\pdv{\theta_{c_2}}{r}\qty(t)}_{L^2\qty(\Omega_2;L_r^2\qty(0,R_\mathrm{s}(\cdot)))}\right). \\
          \end{aligned}
      \end{multline}
      Next, by adding \eqref{eq:DFN_fem_estimation_theta_c1_expand} and \eqref{eq:DFN_fem_estimation_theta_c2_expand}, and selecting $\epsilon_{c_1}$, $\epsilon_{c_2}$ sufficiently small, we have
      \begin{multline*}
          \frac{d}{d t}\left(\left\|\theta_{c_1}(t)\right\|_{L^2\left(\Omega\right)}^2+\left\|\theta_{c_2}(t)\right\|_{L^2\left(\Omega_2 ; L_r^2\left(0, R_\mathrm{s}(\cdot)\right)\right)}^2\right)   
          + \norm{\nabla \theta_{c_1}(t)}_{L^2\qty(\Omega)}^2 + \left\|\frac{\partial \theta_{c_2}}{\partial r}\qty(t)\right\|_{L^2\left(\Omega_2; L_r^2\left(0, R_\mathrm{s}(\cdot)\right)\right)}^2  \\
          \begin{aligned}
          \leq 
          &C h^2\qty(\norm{\phi_1(t)}_{H^2_\mathrm{pw}(\Omega_1)}^2 + \norm{\phi_2(t)}_{H^2_\mathrm{pw}(\Omega_2)}^2 + \norm{{c_1}\qty(t)}_{H_\mathrm{pw}^2\qty(\Omega_1)}^2 + \norm{\pdv{c_1}{t} \qty(t)}_{H_\mathrm{pw}^2\qty(\Omega_1)}^2) \\
          &+ C h^2\left(  \|c_2(t)\|_{H^1\left(\Omega_2; H_r^1\left(0, R_\mathrm{s}(\cdot)\right)\right)}^2 + \left\|\frac{\partial c_2}{\partial t}\qty(t)\right\|_{H^1\left(\Omega_2; H_r^1\left(0, R_\mathrm{s}(\cdot)\right)\right)}^2\right) \\ 
          &+ C \qty(\Delta r)^4\left(\|c_2(t)\|_{L^2\left(\Omega_2; H_r^2\left(0, R_\mathrm{s}(\cdot)\right)\right)}^2+\left\|\frac{\partial c_2}{\partial t}\qty(t)\right\|_{L^2\left(\Omega_2; H_r^2\left(0, R_\mathrm{s}(\cdot)\right)\right)}^2\right) \\ 
          &+ C\qty(\norm{\theta_{c_1}\qty(t)}_{L^2\qty(\Omega)}^2 + \left\|\theta_{c_2}(t)\right\|^2_{L^2\left(\Omega_2; L_r^2\left(0, R_\mathrm{s}(\cdot)\right)\right)}) .\\
          \end{aligned}
      \end{multline*}
      Hence Gronwall's inequality yields,
      \begin{multline*}
        \left\|\theta_{c_1}(t)\right\|_{L^2(\Omega)}^2+\left\|\theta_{c_2}(t)\right\|_{L^2\left(\Omega_2 ; L_r^2\left(0, R_\mathrm{s}(\cdot)\right)\right)}^2 + 
            \int_{0}^{t}\norm{\nabla \theta_{c_1}(s)}_{L^2\qty(\Omega)}^2 \dd s  + \int_{0}^{t} \left\|\frac{\partial \theta_{c_2}}{\partial r}\qty(s)\right\|_{L^2\left(\Omega_2; L_r^2\left(0, R_\mathrm{s}(\cdot)\right)\right)}^2 \dd s \\ 
            \leq C\left(\left\|\theta_{c_1}(0)\right\|_{L^2\left(\Omega\right)}^2+\left\|\theta_{c_2}(0)\right\|_{L^2\qty(\Omega_2 ; L_r^2\qty(0, R_\mathrm{s}\qty(\cdot)))}^2\right) +C\left(h^2+\qty(\Delta r)^4\right).
      \end{multline*}
      Along with \eqref{app_eq:DFN_fem_c1_err_decomp} and \eqref{eq:DFN_c2_surface_estimation}, it follows that 
      \begin{multline*}
          \int_{0}^{t}\left\|c_1(s)-c_{1h}(s)\right\|_{H^1\left(\Omega\right)}^2 \dd s + \int_{0}^{t}\left\|\bar c_2(s)-\bar c_{2h}(s)\right\|_{L^2\left(\Omega_2\right)}^2 \dd s  \\ \leq  C\left(h^2+\qty(\Delta r)^4\right) 
          + C\qty(\left\|\theta_{c_1}(0)\right\|_{L^2\left(\Omega\right)}^2+\left\|\theta_{c_2}(0)\right\|_{L^2\qty(\Omega_2 ; L_r^2\qty(0, R_\mathrm{s}\qty(\cdot)))}^2),
      \end{multline*} 
      \begin{multline*}
          \int_{0}^{t} \left\|c_2(s)-c_{2 h \Delta r}(s)\right\|_{L^2\left(\Omega_2 ; H_r^q\left(0, R_\mathrm{s}(\cdot)\right)\right)}^2 \dd s \\ \leq  C\left(h^2+\qty(\Delta r)^{4-2q}\right)  
          +C\qty(\left\|\theta_{c_1}(0)\right\|_{L^2\left(\Omega\right)}^2+\left\|\theta_{c_2}(0)\right\|_{L^2\qty(\Omega_2 ; L_r^2\qty(0, R_\mathrm{s}\qty(\cdot)))}^2),\quad q=0,1.
      \end{multline*}
      Further by Lemma~\ref{lemma:DFN_fem_estimation_Phi}, we have
      \begin{multline*}
          \int_{0}^{t} \left\|\phi_1(s)-\phi_{1h}(s)\right\|_{H^1\left(\Omega\right)}^2 \dd s + \int_{0}^{t}\left\|\phi_2(s)-\phi_{2h}(s)\right\|_{H^1\left(\Omega_2\right)}^2 \dd s
          \\ \leq  C\left(h^2+\qty(\Delta r)^4\right) 
          +C\qty(\left\|\theta_{c_1}(0)\right\|_{L^2\left(\Omega\right)}^2+\left\|\theta_{c_2}(0)\right\|_{L^2\qty(\Omega_2 ; L_r^2\qty(0, R_\mathrm{s}\qty(\cdot)))}^2).
      \end{multline*} 
      Notice that triangle inequality yields %$\left\|\theta_{c_1}(0)\right\|_{L^2\left(\Omega\right)} \le \norm{c_{10}-c_{1h}^0}_{L^2\left(\Omega\right)} + Ch^2$ and
      \begin{gather*}
        \left\|\theta_{c_1}(0)\right\|_{L^2\left(\Omega\right)} \le \norm{c_{10}-c_{1h}^0}_{L^2\left(\Omega\right)} + Ch, \\ 
        \left\|\theta_{c_2}(0)\right\|_{L^2\qty(\Omega_2 ; L_r^2\qty(0, R_\mathrm{s}\qty(\cdot)))} \le \norm{c_{20}-c_{2h\Delta r}^0}_{L^2\qty(\Omega_2 ; L_r^2\qty(0, R_\mathrm{s}\qty(\cdot)))} + C\qty(h+\qty(\Delta r)^2),
      \end{gather*}
    %   \begin{equation*}
    %       \left\|\theta_{c_2}(0)\right\|_{L^2\qty(\Omega_2 ; L_r^2\qty(0, R_\mathrm{s}\qty(\cdot)))} \le \norm{c_{20}-c_{2h\Delta r}^0}_{L^2\qty(\Omega_2 ; L_r^2\qty(0, R_\mathrm{s}\qty(\cdot)))} + C\qty(h+\qty(\Delta r)^2),
    %   \end{equation*}
      which completes the proof.
    \end{proof}

\section{Numerical experiments}\label{sec:experiments}
    To validate the theoretical analysis presented in this paper, we perform numerical simulations for the DFN model in both 2D+1D and 3D+1D settings, using real battery parameters from \citet{timms_asymptotic_pouch_2021}. All simulations are carried out using our in-house finite element code based on the libMesh library \citep{libMeshPaper}. 
    % Due to computational resource constraints, we report a complete validation of the convergence order for the P3D model, while partial results for the P4D model are provided in section~\ref{app_sec:DFN_exp_3d}.

    \subsection{Case 1: The P3D model}
    \begin{figure}[htbp]
      \centering
      \subfigure[]{
          \includegraphics[width=0.48\textwidth]{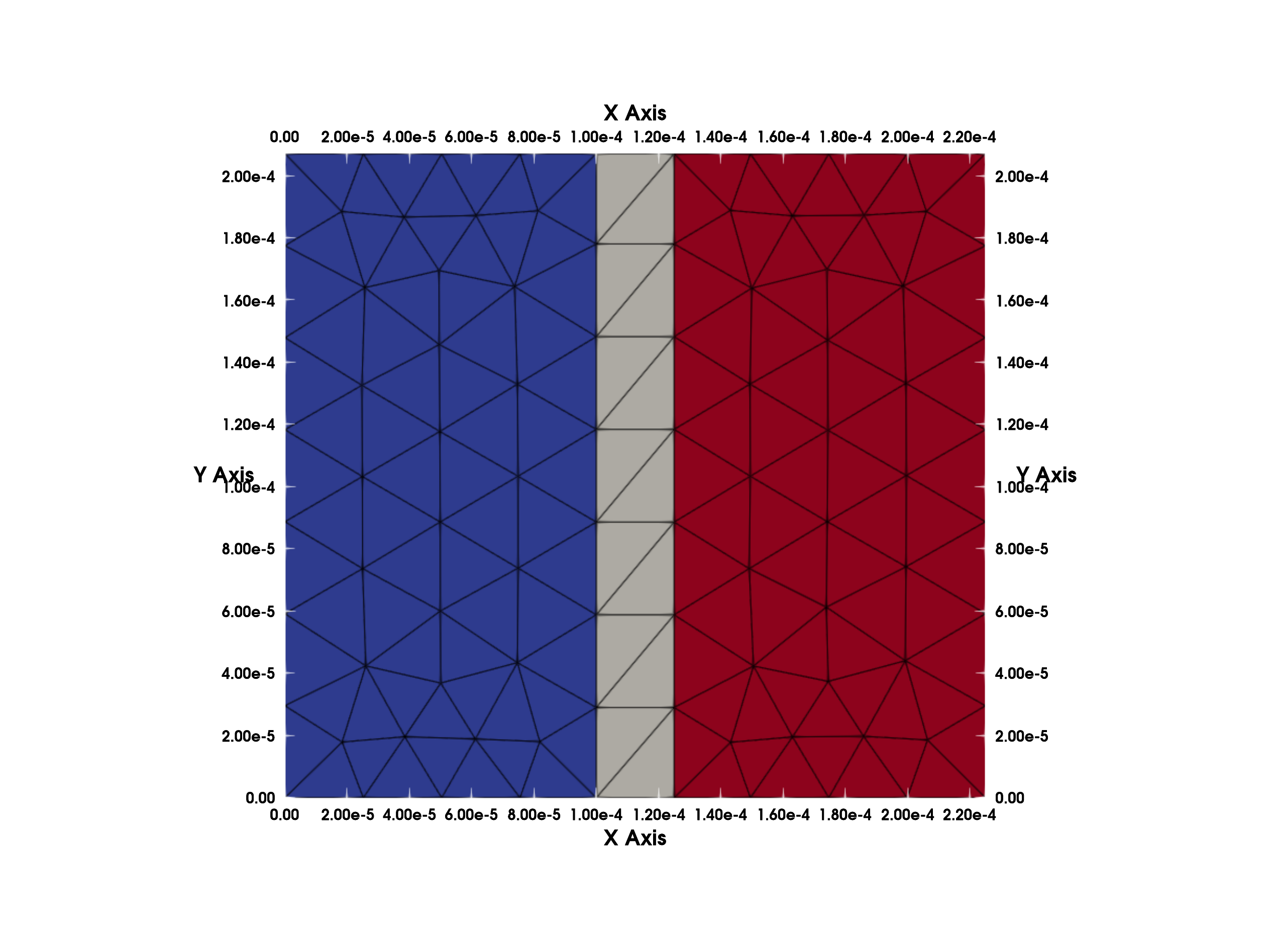}
          \label{fig:DFN_exp_2d_mesh_r0} 
      }%
      \hfill% add desired spacing  
      \subfigure[]{
          \includegraphics[width=0.48\textwidth]{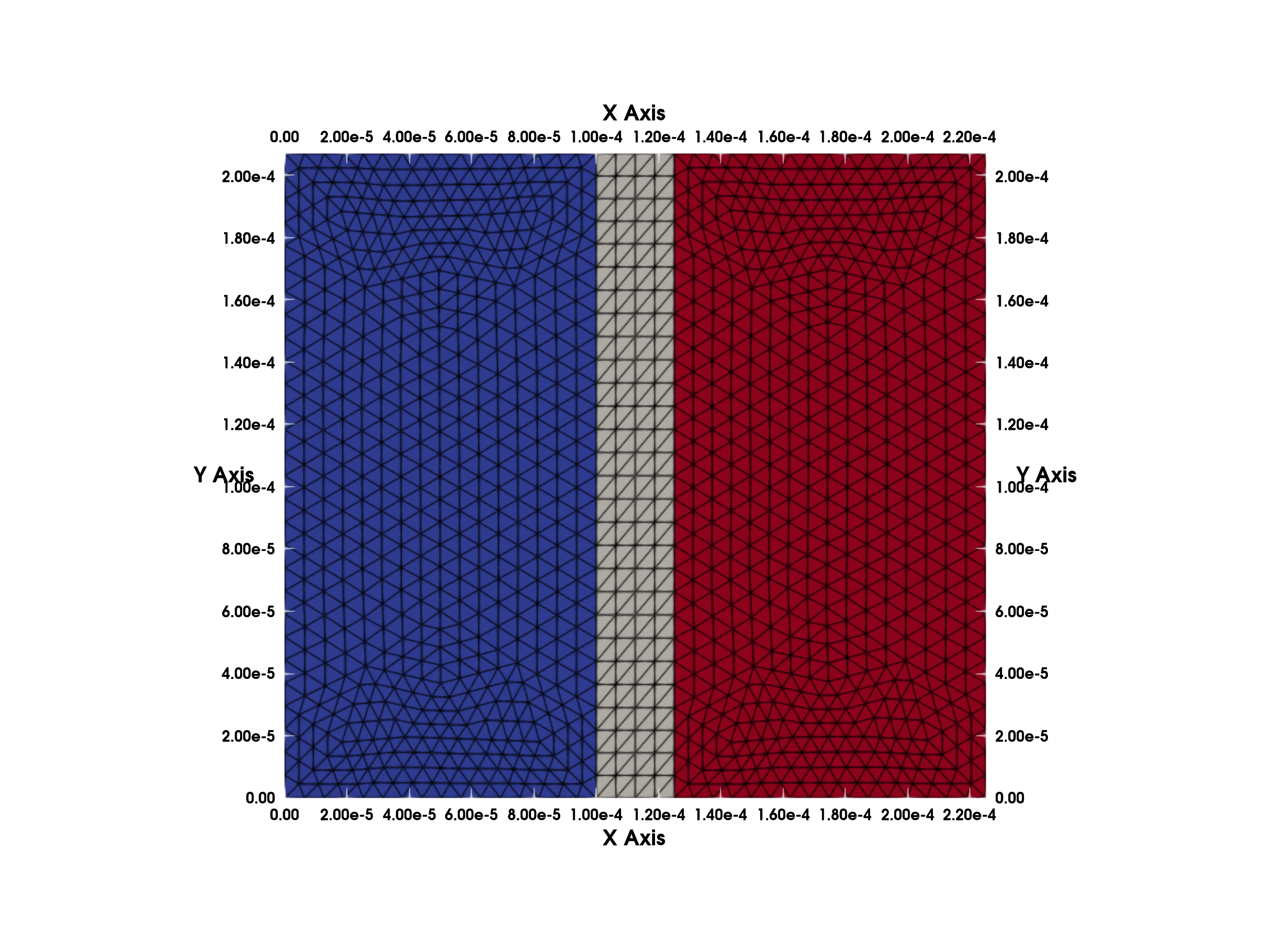}
          \label{fig:DFN_exp_2d_mesh_fine}
      }
      \caption{Spatial meshes for 2d convergence verification. (a) Initial coarse mesh ($R_h=0$). (b) Uniformly refined mesh ($R_h=2$).}
      \label{fig:DFN_exp_2d_mesh} 
    \end{figure}

    In this case, we define the subdomains as follows: $\Omega_\mathrm{n} = [0,100] \times [0,207]$, $\Omega_\mathrm{s} = [100,125] \times [0,207]$, and $\Omega_\mathrm{p} = [125,225] \times [0,207]$ (all lengths in $10^{-6}\mathrm{m}$). The boundary $\Gamma$ is defined as $\Gamma_\mathrm{n} \cup \Gamma_\mathrm{p}$, where $\Gamma_\mathrm{n} = \qty{0} \times \qty[0,207]$ and $\Gamma_\mathrm{p} = \qty{225} \times \qty[0,207]$. A 1C discharge rate is applied.
        
    The initial spatial mesh, shown in Fig.~\ref{fig:DFN_exp_2d_mesh_r0}, is uniformly refined with refinement level $R_h$. 
    The radial grid is initially uniform, with a grid spacing of $\Delta r = 1.25 \times 10^{-6}\mathrm{m}$ and refinement level $R_{\Delta r}$.
    Since the exact solution is unavailable, we use the finite element solution computed on a highly refined mesh ($R_h=5$, $R_{\Delta r} = 5$) as the reference. 
    The influence of temporal discretization via the implicit Euler method is negligible due to the small time step $\Delta t = 0.15625 \mathrm{s}$.

    To assess the convergence with respect to $h$, we fix $R_{\Delta r} = 5$ and refine the spatial mesh from $R_h = 1$ to $R_h = 3$.
    Conversely, to test the convergence with respect to $\Delta r$, we fix $R_h = 5$ and refine the radial grid from $R_{\Delta r} = 1$ to $R_{\Delta r} = 3$.
    The corresponding errors and observed convergence rates at $t_k = k\Delta t$ are reported in Table~\ref{tab:DFN_exp_err_h} and Table~\ref{tab:DFN_exp_err_r}, respectively.
    The numerical results confirm the convergence rates predicted by our theoretical analysis.

    \begin{table*}[h]
        \caption{Error and convergence order for $h$.\label{tab:DFN_exp_err_h}}
        \tabcolsep=0pt%%
        \begin{tabular*}{\textwidth}{@{\extracolsep{\fill}}lcccccccc@{\extracolsep{\fill}}}
            \toprule%
            & \multicolumn{4}{@{}c@{}}{$\norm{\phi_1(\cdot,t_k) - \phi_{1h}^k}_{H^1\qty(\Omega)}$} & \multicolumn{4}{@{}c@{}}{$\norm{\phi_2(\cdot,t_k) - \phi_{2h}^k}_{H^1\qty(\Omega_2)}$} \\
            \cline{2-5}\cline{6-9}%
            k & $R_h=1$ & $R_h=2$ & $R_h=3$ & Rates & $R_h=1$ & $R_h=2$ & $R_h=3$ & Rates \\
            \midrule
             2 & 9.90E-04 & 4.92E-04 & 2.40E-04  & 1.04 & 3.32E-05 & 1.65E-05 & 8.06E-06 & 1.03 \\ 
             4 & 9.90E-04 & 4.90E-04 & 2.39E-04  & 1.04 & 3.32E-05 & 1.65E-05 & 8.06E-06 & 1.03 \\ 
             6 & 1.02E-03 & 5.06E-04 & 2.47E-04  & 1.04 & 3.32E-05 & 1.65E-05 & 8.06E-06 & 1.03 \\ 
             8 & 1.06E-03 & 5.25E-04 & 2.56E-04  & 1.04 & 3.32E-05 & 1.65E-05 & 8.06E-06 & 1.03 \\ 
            10 & 1.05E-03 & 5.19E-04 & 2.53E-04 & 1.04 & 3.32E-05 & 1.65E-05 & 8.06E-06 & 1.03 \\
            \midrule
            & \multicolumn{4}{@{}c@{}}{$\norm{c_1(\cdot,t_k) - c_{1h}^k}_{H^1\qty(\Omega)}$} & \multicolumn{4}{@{}c@{}}{$\norm{\bar c_2(\cdot,t_k) - \bar c_{2h}^k}_{L^2\qty(\Omega_2)}$} \\
            \cline{2-5}\cline{6-9}%
            k & $R_h=1$ & $R_h=2$ & $R_h=3$ & Rates & $R_h=1$ & $R_h=2$ & $R_h=3$ & Rates \\
            \midrule
            2 & 4.24E+00 & 2.25E+00 & 1.11E+00  & 1.02 & 1.36E-02 & 6.77E-03 & 3.31E-03 & 1.04 \\
            4 & 4.70E+00 & 2.37E+00 & 1.16E+00  & 1.03 & 2.96E-03 & 1.47E-03 & 7.18E-04 & 1.04 \\
            6 & 5.70E+00 & 2.89E+00 & 1.42E+00  & 1.03 & 7.91E-03 & 3.94E-03 & 1.92E-03 & 1.03 \\
            8 & 6.80E+00 & 3.42E+00 & 1.68E+00  & 1.03 & 1.73E-02 & 8.63E-03 & 4.21E-03 & 1.04 \\
            10 & 6.83E+00 & 3.40E+00 & 1.66E+00 & 1.04  & 1.21E-02 & 5.99E-03 & 2.92E-03 & 1.03 \\
            \midrule
            & \multicolumn{4}{@{}c@{}}{$\norm{c_2(\cdot,t_k) - c_{2h\Delta r}^k}_{L^2\qty(\Omega_2;H^1_r\qty(0,R_\text{s}\qty(\cdot)))}$} & \multicolumn{4}{@{}c@{}}{$\norm{c_2(\cdot,t_k) - c_{2h\Delta r}^k}_{L^2\qty(\Omega_2;L^2_r\qty(0,R_\text{s}\qty(\cdot)))}$} \\
            \cline{2-5}\cline{6-9}%
            k & $R_h=1$ & $R_h=2$ & $R_h=3$ & Rates & $R_h=1$ & $R_h=2$ & $R_h=3$ & Rates \\
            \midrule
            2 & 2.45E-05 & 1.22E-05 & 5.94E-06  & 1.03 & 3.44E-12 & 1.71E-12 & 8.35E-13 & 1.04 \\
            4 & 9.85E-06 & 4.89E-06 & 2.39E-06  & 1.04 & 2.01E-12 & 9.97E-13 & 4.87E-13 & 1.04 \\
            6 & 1.16E-05 & 5.81E-06 & 2.84E-06  & 1.03 & 2.43E-12 & 1.21E-12 & 5.91E-13 & 1.03 \\
            8 & 2.71E-05 & 1.36E-05 & 6.62E-06  & 1.04 & 4.85E-12 & 2.42E-12 & 1.18E-12 & 1.04 \\
            10 & 1.86E-05 & 9.26E-06 & 4.52E-06 & 1.03  & 4.59E-12 & 2.28E-12 & 1.11E-12 & 1.03 \\
            \botrule
        \end{tabular*}
    \end{table*}

    \begin{table*}[h]
        \caption{Error and convergence order for $\Delta r$.\label{tab:DFN_exp_err_r}}
        \tabcolsep=0pt%%
        \begin{tabular*}{\textwidth}{@{\extracolsep{\fill}}lcccccccc@{\extracolsep{\fill}}}
            \toprule%
            & \multicolumn{4}{@{}c@{}}{$\norm{\phi_1(\cdot,t_k) - \phi_{1h}^k}_{H^1\qty(\Omega)}$} & \multicolumn{4}{@{}c@{}}{$\norm{\phi_2(\cdot,t_k) - \phi_{2h}^k}_{H^1\qty(\Omega_2)}$} \\
            \cline{2-5}\cline{6-9}%
            k & $R_{\Delta r}=1$ & $R_{\Delta r}=2$ & $R_{\Delta r}=3$ & Rates & $R_{\Delta r}=1$ & $R_{\Delta r}=2$ & $R_{\Delta r}=3$ & Rates \\
            \midrule
            2 & 1.33E-05 & 3.04E-06 & 7.09E-07   & 2.10 & 5.52E-08 & 1.28E-08 & 2.99E-09 & 2.10 \\

            4 & 1.86E-05 & 4.58E-06 & 1.08E-06   & 2.08 & 7.75E-08 & 1.89E-08 & 4.46E-09 & 2.08 \\
            
            6 & 5.07E-06 & 1.32E-06 & 3.15E-07   & 2.06 & 2.10E-08 & 5.40E-09 & 1.29E-09 & 2.07 \\
            
            8 & 2.60E-06 & 6.31E-07 & 1.50E-07   & 2.07 & 1.01E-08 & 2.44E-09 & 5.80E-10 & 2.07 \\
            
            10 & 4.99E-06 & 1.31E-06 & 3.14E-07  & 2.06  & 1.89E-08 & 4.98E-09 & 1.19E-09 & 2.06 \\
            \midrule
            & \multicolumn{4}{@{}c@{}}{$\norm{c_1(\cdot,t_k) - c_{1h}^k}_{H^1\qty(\Omega)}$} & \multicolumn{4}{@{}c@{}}{$\norm{\bar c_2(\cdot,t_k) - \bar c_{2h}^k}_{L^2\qty(\Omega_2)}$} \\
            \cline{2-5}\cline{6-9}%
            k & $R_{\Delta r}=1$ & $R_{\Delta r}=2$ & $R_{\Delta r}=3$ & Rates & $R_{\Delta r}=1$ & $R_{\Delta r}=2$ & $R_{\Delta r}=3$ & Rates \\
            \midrule
            2 & 1.87E-02 & 4.34E-03 & 1.02E-03  & 2.10 & 4.20E-03 & 1.01E-03 & 2.38E-04 & 2.08 \\

            4 & 9.75E-03 & 2.49E-03 & 5.93E-04  & 2.07 & 1.94E-03 & 4.41E-04 & 1.05E-04 & 2.08 \\
            
            6 & 2.15E-03 & 5.88E-04 & 1.42E-04  & 2.05 & 1.38E-03 & 3.25E-04 & 7.74E-05 & 2.07 \\
            
            8 & 8.44E-03 & 2.03E-03 & 4.84E-04  & 2.07 & 1.08E-03 & 2.54E-04 & 6.05E-05 & 2.07 \\
            
            10 & 1.09E-02 & 2.75E-03 & 6.57E-04 & 2.07  & 9.64E-04 & 2.28E-04 & 5.44E-05 & 2.07 \\
            \midrule
            & \multicolumn{4}{@{}c@{}}{$\norm{c_2(\cdot,t_k) - c_{2h\Delta r}^k}_{L^2\qty(\Omega_2;H^1_r\qty(0,R_\text{s}\qty(\cdot)))}$} & \multicolumn{4}{@{}c@{}}{$\norm{c_2(\cdot,t_k) - c_{2h\Delta r}^k}_{L^2\qty(\Omega_2;L^2_r\qty(0,R_\text{s}\qty(\cdot)))}$} \\
            \cline{2-5}\cline{6-9}%
            k & $R_{\Delta r}=1$ & $R_{\Delta r}=2$ & $R_{\Delta r}=3$ & Rates & $R_{\Delta r}=1$ & $R_{\Delta r}=2$ & $R_{\Delta r}=3$ & Rates \\
            \midrule
            2 & 1.72E-04 & 8.55E-05 & 4.17E-05  & 1.04 & 2.01E-12 & 5.02E-13 & 1.23E-13 & 2.04 \\
            4 & 1.33E-04 & 6.60E-05 & 3.22E-05  & 1.03 & 1.69E-12 & 4.26E-13 & 1.03E-13 & 2.04 \\
            6 & 1.09E-04 & 5.44E-05 & 2.65E-05  & 1.03 & 1.48E-12 & 3.71E-13 & 9.00E-14 & 2.04 \\
            8 & 9.85E-05 & 4.90E-05 & 2.39E-05  & 1.03 & 1.41E-12 & 3.52E-13 & 8.51E-14 & 2.05 \\
            10 & 8.94E-05 & 4.46E-05 & 2.18E-05 & 1.03  & 1.34E-12 & 3.35E-13 & 8.10E-14 & 2.05 \\
            \botrule
        \end{tabular*}
    \end{table*}

    \subsection{Case 2: The P4D model}\label{app_sec:DFN_exp_3d}

        In this case, we set $\Omega_\mathrm{n} = [0,100] \times [0,207] \times [0,137]$, $\Omega_\mathrm{s} = [100,125] \times [0,207]\times [0,137]$, and $\Omega_\mathrm{p} = [125,225] \times [0,207]\times [0,137]$, with all dimensions in $10^{-6}\mathrm{m}$. The boundary $\Gamma$ is given by $\Gamma_\mathrm{n} \cup \Gamma_\mathrm{p}$, where $\Gamma_\mathrm{n} = \qty{0} \times \qty[0,207] \times [0,137]$ and $\Gamma_\mathrm{p} = \qty{225} \times \qty[0,207]\times [0,137]$. A 1C discharge rate is applied. Temporal discretization effects can be neglected due to the small time step $\Delta t = 0.15625 \mathrm{s}$.

        For the $h$-convergence test, the radial grid is fixed using a non-uniform mesh $\set{1-\frac{1}{2^n}}_{n=1}^{9}\cup \set{0,1}$ (in $10^{-5}\mathrm{m}$). The solution on a fine spatial mesh with refinement level $R_h = 3$ is taken as the reference. Refining the spatial mesh from $R_h = 0$ (Fig.~\ref{fig:DFN_exp_3d_mesh_r0}) to $R_h = 2$ (Fig.~\ref{fig:DFN_exp_3d_mesh_r2}), the errors and observed convergence rates at $t_k = k\Delta t$ are reported in Table~\ref{tab:DFN_exp_3d_err_h}, confirming the expected $\mathcal{O}(h)$ convergence. We remark that the convergence rate is calculated using errors from the $R_h=0$ and $R_h=1$ cases, rather than the last pair ($R_h=1$ and $R_h=2$), as the mesh with $R_h=2$ is already sufficiently close to the reference mesh.

        For the $\Delta r$-convergence test, we fix the spatial mesh to $R_h = 2$ and use an initially uniform radial grid with spacing $\Delta r = 1.25\times10^{-6} \mathrm{m}$. The reference solution is computed on a highly refined radial mesh with $R_{\Delta r} = 5$. The corresponding errors and convergence rates at $t_k = k\Delta t$ are shown in Table~\ref{tab:DFN_exp_3d_err_r}, which are consistent with the theoretical predictions.

        Due to computational resource limitations, we are unable to compute the reference solution on a more refined spatial mesh ($R_h = 4$) or combine $R_h = 3$ with a fine uniform radial grid. Nevertheless, the use of a non-uniform radial mesh for the $h$-convergence test and a moderately refined spatial mesh ($R_h = 2$) for the $\Delta r$-convergence test is sufficient to validate the theoretical convergence rates.

        \begin{figure}[!htbp] 
            \centering
            \subfigure[]{
                \includegraphics[width=0.48\textwidth]{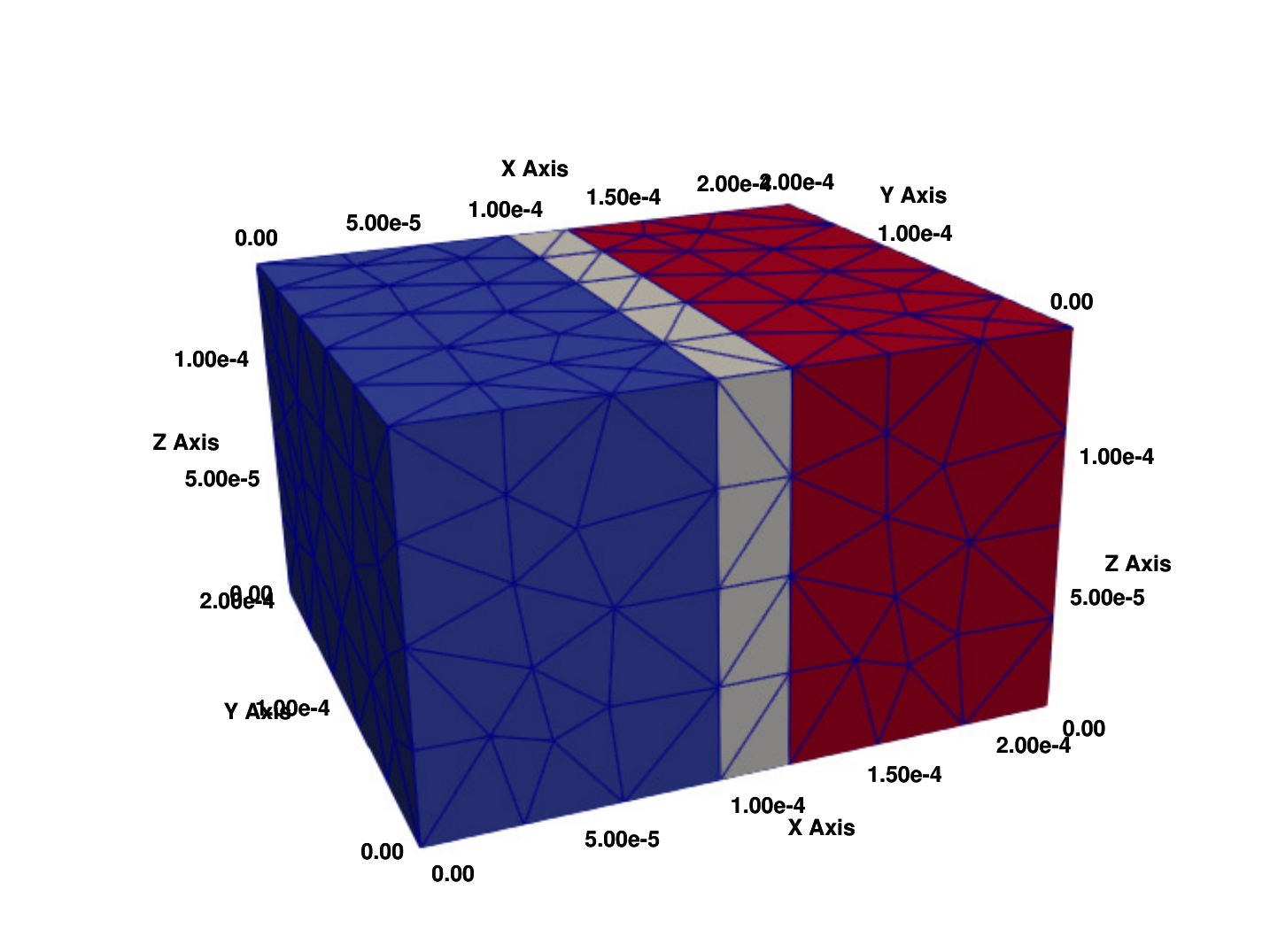}
                \label{fig:DFN_exp_3d_mesh_r0}
            }%
            ~% add desired spacing 
            \subfigure[]{
                \includegraphics[width=0.48\textwidth]{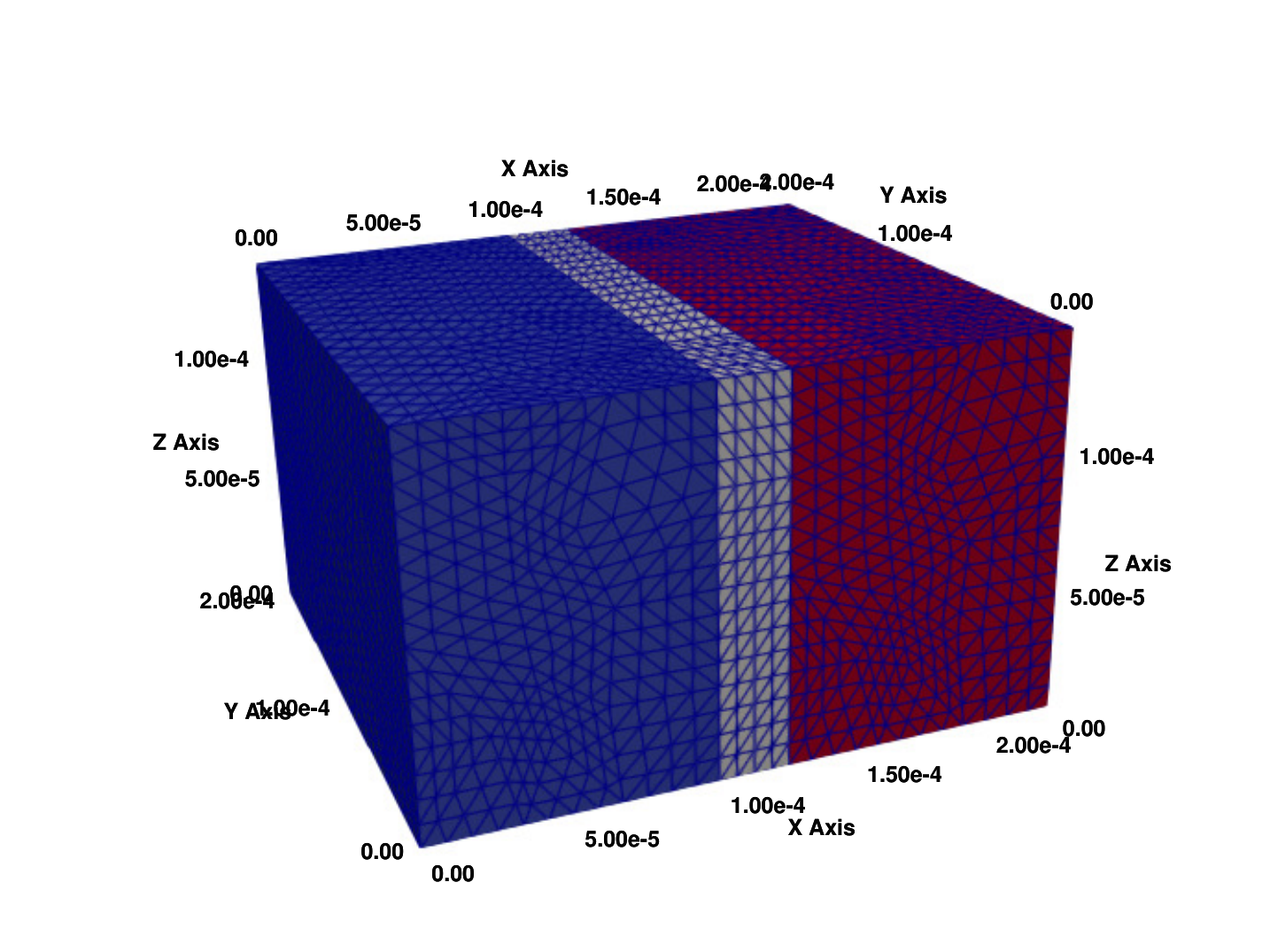}
                \label{fig:DFN_exp_3d_mesh_r2}
            } 
            \caption{Spatial meshes for 3d convergence verification. (a) Initial coarse mesh ($R_h=0$). (b) Uniformly refined mesh ($R_h=2$).}
            \label{fig:DFN_exp_mesh}
        \end{figure}

        \begin{table*}[h]
            \caption{Error and convergence order for $h$.\label{tab:DFN_exp_3d_err_h}}
            \tabcolsep=0pt%%
            \begin{tabular*}{\textwidth}{@{\extracolsep{\fill}}lcccccccc@{\extracolsep{\fill}}}
                \toprule%
                & \multicolumn{4}{@{}c@{}}{$\norm{\phi_1(\cdot,t_k) - \phi_{1h}^k}_{H^1\qty(\Omega)}$} & \multicolumn{4}{@{}c@{}}{$\norm{\phi_2(\cdot,t_k) - \phi_{2h}^k}_{H^1\qty(\Omega_2)}$} \\
                \cline{2-5}\cline{6-9}%
                k & $R_h=0$ & $R_h=1$ & $R_h=2$ & Rates & $R_h=0$ & $R_h=1$ & $R_h=2$ & Rates \\
                \midrule
                 32 & 3.72E-05 & 1.84E-05 & 8.29E-06 & 1.01 & 9.94E-07 & 4.97E-07 & 2.25E-07 & 1.00 \\
                 64 & 4.00E-05 & 1.98E-05 & 8.90E-06 & 1.01 & 9.94E-07 & 4.97E-07 & 2.25E-07 & 1.00 \\
                 96 & 4.20E-05 & 2.08E-05 & 9.32E-06 & 1.02 & 9.94E-07 & 4.97E-07 & 2.25E-07 & 1.00 \\
                128 & 4.36E-05 & 2.15E-05 & 9.67E-06 & 1.02 & 9.94E-07 & 4.97E-07 & 2.25E-07 & 1.00 \\
                \midrule
                & \multicolumn{4}{@{}c@{}}{$\norm{c_1(\cdot,t_k) - c_{1h}^k}_{H^1\qty(\Omega)}$} & \multicolumn{4}{@{}c@{}}{$\norm{\bar c_2(\cdot,t_k) - \bar c_{2h}^k}_{L^2\qty(\Omega_2)}$} \\
                \cline{2-5}\cline{6-9}%
                k & $R_h=0$ & $R_h=1$ & $R_h=2$ & Rates & $R_h=0$ & $R_h=1$ & $R_h=2$ & Rates \\
                \midrule
                 32 & 3.49E-01 & 1.76E-01 & 7.94E-02 & 0.99 & 2.34E-05 & 1.15E-05 & 5.17E-06 & 1.02 \\
                 64 & 4.17E-01 & 2.08E-01 & 9.36E-02 & 1.00 & 3.55E-05 & 1.75E-05 & 7.83E-06 & 1.02 \\
                 96 & 4.63E-01 & 2.30E-01 & 1.03E-01 & 1.01 & 4.59E-05 & 2.26E-05 & 1.01E-05 & 1.02 \\
                128 & 5.03E-01 & 2.49E-01 & 1.12E-01 & 1.02 & 5.54E-05 & 2.73E-05 & 1.22E-05 & 1.02 \\
                \midrule
                & \multicolumn{4}{@{}c@{}}{$\norm{c_2(\cdot,t_k) - c_{2h\Delta r}^k}_{L^2\qty(\Omega_2;H^1_r\qty(0,R_\text{s}\qty(\cdot)))}$} & \multicolumn{4}{@{}c@{}}{$\norm{c_2(\cdot,t_k) - c_{2h\Delta r}^k}_{L^2\qty(\Omega_2;L^2_r\qty(0,R_\text{s}\qty(\cdot)))}$} \\
                \cline{2-5}\cline{6-9}%
                k & $R_h=0$ & $R_h=1$ & $R_h=2$ & Rates & $R_h=0$ & $R_h=1$ & $R_h=2$ & Rates \\
                \midrule
                 32 & 2.46E-07 & 1.21E-07 & 5.43E-08 & 1.02 & 1.08E-13 & 5.35E-14 & 2.39E-14 & 1.02 \\
                 64 & 3.10E-07 & 1.53E-07 & 6.84E-08 & 1.02 & 1.94E-13 & 9.56E-14 & 4.28E-14 & 1.02 \\
                 96 & 3.59E-07 & 1.77E-07 & 7.91E-08 & 1.02 & 2.76E-13 & 1.36E-13 & 6.09E-14 & 1.02 \\
                128 & 3.99E-07 & 1.97E-07 & 8.80E-08 & 1.02 & 3.57E-13 & 1.76E-13 & 7.87E-14 & 1.02 \\
                \botrule
            \end{tabular*}
        \end{table*}

        \begin{table*}[h]
            \caption{Error and convergence order for $\Delta r$.\label{tab:DFN_exp_3d_err_r}}
            \tabcolsep=0pt%%
            \begin{tabular*}{\textwidth}{@{\extracolsep{\fill}}lcccccccc@{\extracolsep{\fill}}}
                \toprule%
                & \multicolumn{4}{@{}c@{}}{$\norm{\phi_1(\cdot,t_k) - \phi_{1h}^k}_{H^1\qty(\Omega)}$} & \multicolumn{4}{@{}c@{}}{$\norm{\phi_2(\cdot,t_k) - \phi_{2h}^k}_{H^1\qty(\Omega_2)}$} \\
                \cline{2-5}\cline{6-9}%
                k & $R_{\Delta r}=1$ & $R_{\Delta r}=2$ & $R_{\Delta r}=3$ & Rates & $R_{\Delta r}=1$ & $R_{\Delta r}=2$ & $R_{\Delta r}=3$ & Rates \\
                \midrule
                2 & 1.54E-07 & 3.54E-08 & 8.24E-09 & 2.10 & 6.43E-10 & 1.49E-10 & 3.48E-11 & 2.10 \\
                4 & 2.16E-07 & 5.31E-08 & 1.26E-08 & 2.08 & 9.01E-10 & 2.19E-10 & 5.19E-11 & 2.08 \\
                6 & 5.75E-08 & 1.53E-08 & 3.66E-09 & 2.06 & 2.40E-10 & 6.28E-11 & 1.50E-11 & 2.07 \\
                8 & 3.02E-08 & 7.28E-09 & 1.73E-09 & 2.07 & 1.19E-10 & 2.85E-11 & 6.77E-12 & 2.07 \\
               10 & 5.55E-08 & 1.51E-08 & 3.63E-09 & 2.07 & 2.10E-10 & 5.76E-11 & 1.38E-11 & 2.06 \\
                
                \midrule
                & \multicolumn{4}{@{}c@{}}{$\norm{c_1(\cdot,t_k) - c_{1h}^k}_{H^1\qty(\Omega)}$} & \multicolumn{4}{@{}c@{}}{$\norm{\bar c_2(\cdot,t_k) - \bar c_{2h}^k}_{L^2\qty(\Omega_2)}$} \\
                \cline{2-5}\cline{6-9}%
                k & $R_{\Delta r}=1$ & $R_{\Delta r}=2$ & $R_{\Delta r}=3$ & Rates & $R_{\Delta r}=1$ & $R_{\Delta r}=2$ & $R_{\Delta r}=3$ & Rates \\
                \midrule
                2 & 2.01E-04 & 4.68E-05 & 1.10E-05 & 2.09 & 4.92E-05 & 1.18E-05 & 2.78E-06 & 2.08 \\
                4 & 9.61E-05 & 2.49E-05 & 5.94E-06 & 2.07 & 2.27E-05 & 5.16E-06 & 1.22E-06 & 2.08 \\
                6 & 1.49E-05 & 3.92E-06 & 9.52E-07 & 2.04 & 1.68E-05 & 3.80E-06 & 9.05E-07 & 2.07 \\
                8 & 7.40E-05 & 1.78E-05 & 4.23E-06 & 2.07 & 1.31E-05 & 2.98E-06 & 7.08E-07 & 2.07 \\
               10 & 1.12E-04 & 2.87E-05 & 6.86E-06 & 2.07 & 1.17E-05 & 2.67E-06 & 6.37E-07 & 2.07 \\
                
                \midrule
                & \multicolumn{4}{@{}c@{}}{$\norm{c_2(\cdot,t_k) - c_{2h\Delta r}^k}_{L^2\qty(\Omega_2;H^1_r\qty(0,R_\text{s}\qty(\cdot)))}$} & \multicolumn{4}{@{}c@{}}{$\norm{c_2(\cdot,t_k) - c_{2h\Delta r}^k}_{L^2\qty(\Omega_2;L^2_r\qty(0,R_\text{s}\qty(\cdot)))}$} \\
                \cline{2-5}\cline{6-9}%
                k & $R_{\Delta r}=1$ & $R_{\Delta r}=2$ & $R_{\Delta r}=3$ & Rates & $R_{\Delta r}=1$ & $R_{\Delta r}=2$ & $R_{\Delta r}=3$ & Rates \\
                \midrule
                2 & 2.01E-06 & 1.00E-06 & 4.88E-07 & 1.04 & 2.36E-14 & 5.88E-15 & 1.43E-15 & 2.04 \\
                4 & 1.55E-06 & 7.72E-07 & 3.77E-07 & 1.03 & 1.98E-14 & 4.99E-15 & 1.21E-15 & 2.04 \\
                6 & 1.28E-06 & 6.36E-07 & 3.10E-07 & 1.03 & 1.73E-14 & 4.35E-15 & 1.05E-15 & 2.04 \\
                8 & 1.15E-06 & 5.74E-07 & 2.80E-07 & 1.03 & 1.64E-14 & 4.12E-15 & 9.96E-16 & 2.05 \\
               10 & 1.05E-06 & 5.22E-07 & 2.55E-07 & 1.03 & 1.57E-14 & 3.92E-15 & 9.48E-16 & 2.05 \\
               
                \botrule
            \end{tabular*}
        \end{table*}

%%%%%%%%%%%%%%

\appendix

\section{Proof of Lemma~\ref{lemma:DFN_fem_estimation_Phi}}\label{app_sec:DFN_err_phi}
    \begin{proof}
    Let $\Phi = \qty(\phi_1, \phi_2)$, $\Psi = \qty(\varphi_1, \varphi_2) \in H^1_*(\Omega)\times H^1(\Omega_2)=: V$,  and the dual space of $V$ is denoted by $V^*$. Define the operators 
    \begin{gather*}
        B\colon V \rightarrow V^*,\quad  \langle B(\Phi), \Psi\rangle :=\sum_{m \in \qty{ \mathrm{n,p}}}\int_{\Omega_m} a_2 J_m\left(c_{1}, \bar c_{2}, \eta \right)  \qty(\varphi_2 - \varphi_1) \,\dd x,\\ 
        D \in V^*, \quad D(\Psi) := -\int_{\Omega} \kappa_2 \nabla f\left(c_1\right) \cdot \nabla \varphi_1 \, \dd x + \int_\Gamma I \varphi_2 \, \dd x,
    \end{gather*}
    and the bilinear form 
    \begin{displaymath}
        a\colon V\times V \rightarrow \mathbb{R}, \quad a(\Phi, \Psi):= \int_{\Omega} \kappa_1\nabla \phi_1 \cdot \nabla \varphi_1 \, \dd x + \int_{\Omega_2}\sigma \nabla \phi_2 \cdot \nabla \varphi_2 \, \dd x.
    \end{displaymath}
    Then \eqref{eq:DFN_weak_phi1} and \eqref{eq:DFN_weak_phi2} could be reformulated into the equation:
    \begin{equation}\label{app_eq:DFN_weak_Phi}
        a\left(\Phi, \Psi\right)+\langle B\qty(\Phi), \Psi\rangle+D\left(\Psi\right)=0, \quad \forall \Psi \in V.
    \end{equation}
    For the finite element discretization, let $\Phi_h = \qty(\phi_{1h}, \phi_{2h}),\; \Psi_h = \qty(\varphi_{1h}, \varphi_{2h}) \in W_h (\bar \Omega)\times V^{(1)}_h(\bar \Omega_2)=:V_h$, and $V_h^{*}$ be the dual space of $V_h$. %$ U_h = \sum_{m\in \set{\mathrm{n,p}}} U_m(c_{1h},\bar c_{2h})\boldsymbol{1}_{\Omega_m}$,  $\eta_{h} = \phi_{2h} - \phi_{1h} - U_h$,
    Likewise, we define
    \begin{gather*}
        B_h\colon V_h \rightarrow V_h^{*},\quad \langle B_h(\Phi_h), \Psi_h\rangle= \sum_{m \in \qty{ \mathrm{n,p}}}\int_{\Omega_m} a_2 J_m\left(c_{1h}, \bar c_{2h}, \eta_h \right) \qty(\varphi_{2h} - \varphi_{1h})\, \dd x, \\ 
        D_h \in V_h^{*}, \quad D_h(\Psi_h) = -\int_{\Omega} \kappa_{2h} \nabla f\left(c_{1h}\right) \cdot \nabla \varphi_{1h} \, \dd x + \int_\Gamma I \varphi_{2h} \, \dd x,
    \end{gather*}
    the bilinear form $a_h\colon V_h\times V_h \rightarrow \mathbb{R}$, 
    \begin{displaymath}
        a_h(\Phi_h, \Psi_h):= \int_{\Omega} \kappa_{1h}\nabla \phi_{1h} \cdot \nabla \varphi_{1h} \, \dd x + \int_{\Omega_2}\sigma \nabla \phi_{2h} \cdot \nabla \varphi_{2h} \, \dd x.
    \end{displaymath}
    \eqref{eq:DFN_fem_phi1} and \eqref{eq:DFN_fem_phi2} could also be reformulated as
    \begin{equation}\label{app_eq:DFN_fem_Phi}
    a_h\left(\Phi_h, \Psi_h\right)+\left\langle B_h\left(\Phi_h\right), \Psi_h\right\rangle+D_h\left(\Psi_h\right)=0, \quad \forall \Psi_h \in V_h.
    \end{equation} 
    Since $V_h \subset V$, \eqref{app_eq:DFN_weak_Phi} implies
    \begin{equation}\label{app_eq:DFN_weak_Phi_test_h}
    a\left(\Phi, \Psi_h\right)+\langle B\qty(\Phi), \Psi_h\rangle+D\left(\Psi_h\right)=0, \quad \forall \Psi_h \in V_h.
    \end{equation}
    Define the projection operator $P_{h}^1$ from $H_*^1(\Omega)$ to $W_h(\bar \Omega)$ such that $\forall \phi_1 \in H_*^1(\Omega)$,
    \begin{equation*}
    \int_{\Omega} \kappa_1 \nabla\left(\phi_1-P_{h}^1 \phi_1\right) \cdot \nabla \varphi_h\, \dd x =0, \quad \forall \varphi_h \in W_h\qty(\bar \Omega),
    \end{equation*}
    and the projection operator $P_{h}^2$ from $H^1(\Omega_2)$ to $V_h^{(1)}(\bar \Omega_2)$ such that $\forall \phi_2 \in H^1( \Omega_2)$,
    \begin{equation*}
    \int_{\Omega_2}  \sigma \nabla\left(\phi_2-P_{h}^2 \phi_2\right) \cdot \nabla \varphi_h\, \dd x+ \int_{\Omega_2} \underline{\sigma} \left(\phi_2-P_{h}^2 \phi_2\right) \varphi_h \, \dd x=0, \quad \forall \varphi_h \in V_h^{(1)}(\bar \Omega_2).
    \end{equation*}
    Then, we can decompose the finite element approximation error as
    \begin{equation}\label{app_eq:DFN_fem_phi_err_decomp}
    \begin{aligned}
    & \phi_1-\phi_{1 h}=\left(\phi_1-P_h^{1} \phi_1\right)+\left(P_h^{1} \phi_1-\phi_{1 h}\right)=:\rho_1+\theta_1, \\
    & \phi_2-\phi_{2 h}=\left(\phi_2-P_h^2 \phi_2\right)+\left(P_h^2 \phi_2-\phi_{2 h}\right)=:\rho_2+\theta_2. \\
    \end{aligned}
    \end{equation}
    Setting $\rho = \qty(\rho_1,\rho_2)$, $\theta = \qty(\theta_1,\theta_2)$, we have $\Phi-\Phi_h=\rho+\theta$.
    Under the Assumption~\ref{asp:DFN_fem_regularity}, we have the error estimates
    \begin{equation}\label{app_eq:DFN_fem_phi_projection_error}
    \norm{\rho_1}_{H^1\qty(\Omega)}  \le Ch\norm{\phi_1}_{H^2_\mathrm{pw}\qty(\Omega_1)}\qcomma{} \norm{\rho_2}_{H^1\qty(\Omega_2)}  \le Ch\norm{\phi_2}_{H^2_\mathrm{pw}\qty(\Omega_2)}.
    \end{equation} 
    Therefore, once the estimates of $\theta_i$ are given, the finite element error would be obtained by \eqref{app_eq:DFN_fem_phi_err_decomp}.
    % \begin{lemma}\label{app_lemma:DFN_fem_estimation_Phi}
    % For $t\in \qty[0,T_\mathrm{end}]$ a.e., $\phi_1(t)\in H^2_\mathrm{pw}(\Omega_1) $, $\phi_2(t )\in H^2_\mathrm{pw}(\Omega_2) $, there is a constant $C$ that does not depend on $h$ and $\Delta r$, such that
    % \begin{multline}\label{app_eq:DFN_fem_estimation_Phi}
    %     \norm{\Phi(t) - \Phi_{h}(t)}_{V}^2  \le Ch^2\qty(\norm{\phi_1(t)}_{H^2_\mathrm{pw}(\Omega_1)}^2 + \norm{\phi_2(t)}_{H^2_\mathrm{pw}(\Omega_2)}^2) + \\ 
    %     C\qty( \norm{c_1(t) - c_{1h}(t)}_{H^1(\Omega_1)}^2 + \norm{\bar c_2(t) - \bar c_{2h}(t)}_{L^2(\Omega_2)}^2 ).
    % \end{multline}
    % \end{lemma}
    
    Setting $P_h \Phi := \qty(P_h^1 \phi_1,P_h^2 \phi_2)$, substract \eqref{app_eq:DFN_fem_Phi} from \eqref{app_eq:DFN_weak_Phi_test_h} and by virtue of \eqref{app_eq:DFN_fem_phi_err_decomp},
    \begin{multline}\label{app_eq:DFN_theta_eqn}
            a_h\left(\theta, \Psi_h\right)+\langle B\left(P_h \Phi\right)-B\left(\Phi_h\right), \Psi_h\rangle
             =\left\langle B\left(P_h \Phi\right)-B(\Phi), \Psi_h\right\rangle+\left\langle B_h\left(\Phi_h\right)-B\left(\Phi_h\right), \Psi_h\right\rangle \\
             +\qty(a_h\left(P_h \Phi, \Psi_h\right) - a\left(P_h \Phi, \Psi_h\right)) +\qty(D_h\left(\Psi_h\right) -D\left(\Psi_{h}\right)) - a\left(\rho, \Psi_h\right).
    \end{multline}
    For any $\Phi = \qty(\phi_1, \phi_2),\:\tilde\Phi = \qty(\tilde\phi_1, \tilde\phi_2)\in V$, setting $\tilde{\eta}= \tilde\phi_2 - \tilde\phi_1 - U$, 
    \begin{align*}
          \langle B(\Phi)-B(\tilde{\Phi}), \Phi-\tilde{\Phi}\rangle 
        = &\sum_{m \in \qty{ \mathrm{n,p}}}\int_{\Omega_m} a_2\qty(J_m\left(c_1, \bar c_2, \eta\right) - J_m\qty(c_1, \bar{c}_2, \tilde\eta))\qty[\left(\phi_2-\tilde{\phi}_2\right)-\left(\phi_1-\tilde{\phi}_1\right)] \, \dd x \nonumber \\
        = &\sum_{m \in \qty{ \mathrm{n,p}}}\int_{\Omega_m} a_2 \frac{\partial J_m}{\partial \eta}\left(c_1, \bar{c}_2, \xi\right)\left[\left(\phi_2-\tilde{\phi}_2\right)-\left(\phi_1-\tilde{\phi}_1\right)\right]^2 \, \dd x \nonumber\\
        \ge& C \int_{\Omega_2} \left[\left(\phi_2-\tilde{\phi}_2\right)-\left(\phi_1-\tilde{\phi}_1\right)\right]^2  \, \dd x \nonumber\\
        = &C\left(\int_{\Omega_2}\left|\phi_2-\tilde{\phi}_2\right|^2+\left|\phi_1-\tilde{\phi}_1\right|^2\, \dd x - 2 \int_{\Omega_2}\left(\phi_2-\tilde\phi_2\right)\left(\phi_1-\tilde\phi_1\right)\, \dd x\right)  \nonumber\\
        \ge & C \qty(\qty(1-2\epsilon)\int_{\Omega_2}\left|\phi_2-\tilde{\phi}_2\right|^2 \, \dd x + \qty(1- \frac{1}{2\epsilon})\int_{\Omega_2}\left|\phi_1-\tilde{\phi}_1\right|^2 \, \dd x)  \, \dd x.  
    \end{align*}  
    Moreover, by Assumption~\ref{asp:DFN_fem_prior_bound}, $c_{1h}$ is uniformly bounded away from zero, so that $\kappa_{1h} =\kappa_1\qty(c_{1h})\ge C > 0$. Hence, 
    \begin{multline*}
        a_h\left(\Phi-\tilde{\Phi}, \Phi-\tilde{\Phi}\right)+\langle B(\Phi)-B(\tilde{\Phi}), \Phi-\tilde{\Phi}\rangle   
        \ge C_1\norm{\nabla \qty(\phi_1 - \tilde\phi_{1})}_{L^2(\Omega)}^2 \\+ \underline{\sigma} \norm{\nabla \qty(\phi_2 - \tilde\phi_{2})}_{L^2(\Omega_2)}^2 + C_2\qty(1- \frac{1}{2\epsilon}) \norm{\phi_1 - \tilde\phi_{1}}_{L^2(\Omega)}^2   + C_2 \qty(1-2\epsilon) \norm{\phi_2 - \tilde\phi_{2}}_{L^2(\Omega_2)}^2. 
    \end{multline*}  
    By Poincar\'e inequality, $\norm{\phi_1 - \tilde\phi_{1}}_{L^2(\Omega)}\le C_p\norm{\nabla \qty(\phi_1 - \tilde\phi_{1})}_{L^2(\Omega)} $. Selecting $\epsilon <\frac{1}{2}$ sufficiently large, we then have
    \begin{equation*}
            a_h\left(\Phi-\tilde{\Phi}, \Phi-\tilde{\Phi}\right)+\langle B(\Phi)-B(\tilde{\Phi}), \Phi-\tilde{\Phi}\rangle 
            \ge C\qty( \norm{\nabla\qty(\phi_1 - \tilde\phi_{1})}^2_{L^2(\Omega)} + \norm{\phi_2 - \tilde\phi_{2}}^2_{H^1(\Omega_2)}). 
    \end{equation*} 
    Therefore, when taking $\Psi_h = \theta$ in \eqref{app_eq:DFN_theta_eqn}, on the left-hand side, 
    \begin{equation*}
        a_h\left(\theta, \theta\right)+\langle B\left(P_h \Phi\right)-B\left(\Phi_h\right), P_h \Phi - \Phi_h\rangle \ge C \norm{\theta}^2_V.
    \end{equation*}
    Since Assumption~\ref{asp:nonlinear_function} ensures that the nonlinear functions $J_m$, $U_m$, $\kappa_i$ and $f'$ are bouded and Lipschitz continuous on any bounded interval, and Assumptions~\ref{asp:DFN_prior_bound} and \ref{asp:DFN_fem_prior_bound} guarantee  the solution and its approximation remain within such intervals, it follows that the terms on the right-hand side of \eqref{app_eq:DFN_theta_eqn} satisfy the estimate
    \begin{equation*}
        \begin{aligned}
        \langle B(P_h\Phi)-B(\Phi), \theta\rangle 
        = & \sum_{m \in \qty{ \mathrm{n,p}}} \int_{\Omega_m} a_2\left(J_m \left( c_1, \bar c_2, P_h^2\phi_2 - P_h^1\phi_1 - U\right) - J_m\left(c_1, \bar c_2, \eta \right)\right)\left(\theta_2-\theta_1\right)\, \dd x \\
        \leq & C\left(\left\|\phi_1-P_h^1{\phi}_1\right\|_{L^2(\Omega)}+\left\|\phi_2-P_h^2{\phi}_2\right\|_{L^2(\Omega_2)}\right)\|\theta\|_{{L^2(\Omega)} \times L^2(\Omega_2)} \\
        % \leq & C\|\Phi-P_h{\Phi}\|_{{L^2(\Omega_1)} \times L^2(\Omega_2)} \|\theta\|_{{L^2(\Omega_1)} \times L^2(\Omega_2)} \\
        \leq & Ch\qty(\norm{\phi_1}_{H^2_\mathrm{pw}(\Omega_1)}+ \norm{\phi_2}_{H^2_\mathrm{pw}(\Omega_2)}) \|\theta\|_{{L^2(\Omega)} \times L^2(\Omega_2)},
        \end{aligned} 
    \end{equation*}
    \begin{equation*}
        \begin{aligned}
            \langle B(\Phi_h)-B_h(\Phi_h), \theta \rangle =&\sum_{m \in \qty{ \mathrm{n,p}}} \int_{\Omega_m}  a_2\qty(J_m \left( c_1, \bar c_2, \phi_{2h} -  \phi_{1h} -U \right) -J_m\left( c_{1h}, \bar c_{2h}, \eta_h \right)) \left(\theta_{2}-\theta_{1}\right) \, \dd x\\
       \leq &C\left( \left\|c_1-c_{1h}\right\|_{L^2(\Omega_2)}+\left\|\bar c_2-\bar c_{2h}\right\|_{L^2(\Omega_2)} + \norm{U-U_h}_{L^2(\Omega_2)} \right)\|\theta\|_{{L^2(\Omega_1)} \times L^2(\Omega_2)} \\
       \leq &  C \qty( \left\|c_1-c_{1h}\right\|_{L^2(\Omega)}+\left\|\bar c_2-\bar c_{2h}\right\|_{L^2(\Omega_2)} ) \|\theta\|_{{L^2(\Omega)} \times L^2(\Omega_2)},
        \end{aligned}
    \end{equation*}
    \begin{equation*}
        \begin{aligned}
        a_h\left(P_h\Phi, \theta\right) - a\left(P_h\Phi, \theta\right) = &\int_{\Omega}\kappa_{1h} \nabla \qty(P_h^{1} \phi_1 - \phi_{1 }) \cdot \nabla \theta_{1} \, \dd x + \int_{\Omega}\left(\kappa_{1h}-\kappa_{1}\right) \nabla \phi_{1} \cdot \nabla \theta_{1} \, \dd x \\
        & + \int_{\Omega} \kappa_{1} \nabla \qty( \phi_{1} - P_h^{1} \phi_1 )\cdot \nabla \theta_{1} \, \dd x\\ 
        \leq & Ch\norm{\phi_1}_{H^2_\mathrm{pw}(\Omega_1)}\left\|\nabla \theta_1\right\|_{L^2(\Omega) } + C\left\|c_1-c_{1 h}\right\|_{L^2(\Omega) } \left\|\nabla \theta_1\right\|_{L^2(\Omega) },
        \end{aligned} 
    \end{equation*}
    \begin{equation}\label{app_eq:DFN_err_phi_D_error}
        \begin{aligned}
            \left|D(\theta)-D_h(\theta)\right|=&\left|\int_{\Omega}\qty[\kappa_2 \nabla f\left(c_1\right)-\kappa_{2h} \nabla f\left(c_{1h}\right)] \cdot \nabla \theta_1\right| \, \dd x\\
            \leq & \int_{\Omega}\left|\kappa_2 -\kappa_{2h}\right|\left|\nabla f\left(c_1\right)\right|\qty|\nabla \theta_1|\, \dd x  
            +\int_{\Omega}\left|\kappa_{2h}\right|\qty|f'\qty(c_1)- f'\qty(c_{1h})| \left|\nabla c_1\right||\nabla \theta_1|\, \dd x \\ 
            & +\int_{\Omega}\left|\kappa_{2h}\right| \qty| f'\qty(c_{1h})|  \left|\nabla c_{1} - \nabla c_{1 h}\right||\nabla \theta_1| \, \dd x\\
            \leq& C\left\|c_1-c_{1h} \right\|_{H^1(\Omega)}\|\nabla \theta_1\|_{L^2(\Omega) }.
        \end{aligned}
    \end{equation}
    We remark that the derivation of \eqref{app_eq:DFN_err_phi_D_error} does not require any assumption on the $L^{\infty}$-boundedness of $\nabla c_{1h}$.
    %It is worth emphasizing that no assumption on the $L^{\infty}$-boundedness of $\nabla c_{1h}$ is required in the derivation of \eqref{app_eq:DFN_err_phi_D_error}. 
    In addition, we have
    \begin{equation*}
        \begin{aligned} 
            a\left(\rho, \theta\right)& =\int_{\Omega} \kappa_1 \nabla\left(\phi_1-P_h^1 \phi_1\right) \cdot \nabla \theta_1 \, \dd x +\int_{\Omega_2} \sigma \nabla\left(\phi_2-P_h^2 \phi_2\right) \cdot \nabla \theta_2  \, \dd x\\
            & =-\int_{\Omega_2} \underline{\sigma}\left(\phi_2-P_h^2 \phi_2\right) \cdot \theta_2 \, \dd x \\
            & \le Ch\norm{\phi_2}_{H^2_\mathrm{pw}\qty(\Omega_2)}\norm{\theta_2}_{L^2(\Omega_2) }.
        \end{aligned}
    \end{equation*}
    By the Young's inequality with the coefficients before $\|\theta\|_{{L^2(\Omega)} \times L^2(\Omega_2)}$, $\norm{\theta_2}_{L^2(\Omega_2)}$ and  $\|\nabla \theta_1\|_{L^2(\Omega) }$ sufficiently small, we have
    \begin{equation*}
        \norm{\theta}_{V}^2  \le Ch^2\qty(\norm{\phi_1}_{H^2_\mathrm{pw}(\Omega_1)}^2 + \norm{\phi_2}_{H^2_\mathrm{pw}(\Omega_2)}^2)
        + C\qty( \norm{c_1 - c_{1h}}_{H^1(\Omega)}^2 + \norm{\bar c_2 - \bar c_{2h}}_{L^2(\Omega_2)}^2 ).
    \end{equation*}
    Finally, the lemma is proved by \eqref{app_eq:DFN_fem_phi_err_decomp} and \eqref{app_eq:DFN_fem_phi_projection_error}.
    \end{proof}

\section{Proof of Lemma~\ref{lemma:DFN_fem_estimation_theta_c1}}\label{app_sec:DFN_err_theta_c1}
    
    \begin{proof}
        Using \eqref{eq:DFN_weak_c1}, \eqref{eq:DFN_fem_c1}, \eqref{eq:DFN_projection_c1} and \eqref{app_eq:DFN_fem_c1_err_decomp}, we have
        \begin{equation*}
                \int_{\Omega}  \frac{\partial \theta_{c_1}}{\partial t} v_h \, \dd x + \int_{\Omega} k_1 \nabla \theta_{c_1} \cdot \nabla v_h \, \dd x =\int_{\Omega_2} a_1\left(J-J_h\right) v_h  \, \dd x 
                -\int_{\Omega} \frac{\partial \rho_{c_1}}{\partial t} v_h \, \dd x + \int_{\Omega}  \rho_{c_1} v_h \, \dd x.
        \end{equation*}
        Taking $v_h=\theta_{c_1}$, Young's inequality and \eqref{eq:DFN_estimation_J} yield an arbitrarily small number $\epsilon > 0$, such that
        \begin{multline*}
            \frac{1}{2}\dv{}{t}\norm{\theta_{c_1}}_{L^2\qty(\Omega)}^2 + \underline{k_1} \norm{\nabla \theta_{c_1}}_{L^2\qty(\Omega)}^2 \le \norm{\rho_{c_1}}_{L^2\qty(\Omega)}^2 +  \norm{\pdv{\rho_{c_1}}{t}}_{L^2\qty(\Omega)}^2 + C\qty(\epsilon) \norm{\theta_{c_1}}_{L^2\qty(\Omega)}^2 \\ 
            + \epsilon\qty( \norm{\phi_1-\phi_{1h}}^2_{L^2\qty(\Omega)} + \norm{\phi_2-\phi_{2h}}^2_{L^2\qty(\Omega_2)} + 
            \norm{c_1 - c_{1h}}_{L^2\qty(\Omega)}^2 + \norm{\bar c_2 - \bar c_{2h}}_{L^2\qty(\Omega_2)}^2).
        \end{multline*} 
        Hence, \eqref{eq:DFN_fem_estimation_theta_c1} follows by Assumption~\ref{app_eq:DFN_fem_c1_err_decomp}, \eqref{eq:DFN_fem_c1_projection_error} and \eqref{eq:DFN_fem_c1t_projection_error}.
    \end{proof} 

\section*{Acknowledgments}
The authors would like to thank Jiming Wu, Institute of Applied Physics and Computational Mathematics, for discussion regarding optimal-order convergence. The computations were done on the high performance computers of State Key Laboratory of Scientific and Engineering Computing, Chinese Academy of Sciences. 
\section*{Funding}
This work was funded by the National Natural Science Foundation of China (grant 12371437) and the Beijing Natural Science Foundation (grant Z240001).

% \bibliographystyle{plain} 
% \bibliography{bib/PDE.bib, bib/FEM.bib, bib/review.bib, bib/DFN.bib, bib/DFN_heat.bib, bib/DFN_mechanics.bib}

%USE THE BELOW OPTIONS IN CASE YOU NEED AUTHOR YEAR FORMAT.
\bibliographystyle{apalike}
\bibliography{bib/PDE.bib, bib/FEM.bib, bib/review.bib, bib/DFN.bib, bib/DFN_heat.bib, bib/DFN_mechanics.bib}

\end{document}